\documentclass[10pt]{amsart}

\usepackage{amssymb}
\usepackage{amsmath}
\usepackage{enumerate}
\usepackage{amsfonts}
\usepackage{color}
\usepackage{latexsym}

\usepackage[pdftex]{graphicx}
\DeclareGraphicsExtensions{.pdf,.png,.jpg}
\usepackage{xcolor}
\usepackage{tikz}
\usetikzlibrary{arrows}

\allowdisplaybreaks
\newtheorem{thm}{Theorem}[section]
\newtheorem{coro}[thm]{Corollary}
\newtheorem{conj}[thm]{Conjecture}
\newtheorem{lema}[thm]{Lemma}
\newtheorem{propo}[thm]{Proposition}
\newtheorem{remark}[thm]{Remark}

\numberwithin{equation}{section}

\renewcommand{\a}{\alpha}

\renewcommand{\b}{\beta}

\newcommand{\RR}{\mathbb{R}}

\allowdisplaybreaks

\author[\'O.\ Ciaurri]{\'Oscar Ciaurri}
\address[\'O.\ Ciaurri]{Departamento de Matem\'aticas y Computaci\'on,
         Universidad de La Rioja,
         26004 Logro\~no, Spain}
\email{oscar.ciaurri@unirioja.es}

\author[A.\ Nowak]{Adam Nowak}
\address[A.\ Nowak]{Institute of Mathematics,
Polish Academy of Sciences,
\'Sniadeckich 8,
00-656 Warszawa, Poland}
\email{adam.nowak@impan.pl}

\author[L.\ Roncal]{Luz Roncal}
\address[L.\ Roncal]{BCAM - Basque Center for Applied Mathematics,
48009 Bilbao, Spain and
Ikerbasque, Basque Foundation for Science, 48011 Bilbao, Spain}
\email{lroncal@bcamath.org}

\thanks{The first-named author was supported by the grant PGC2018-096504-B-C32 from Spanish Government.
The second-named author was supported by the National Science Centre of Poland within the research
project OPUS 2017/27/B/ST1/01623. The third-named author was supported by the Basque Government through
BERC 2018--2021 program, by Spanish Ministry of Economy and Competitiveness MINECO through BCAM Severo Ochoa
excellence accreditation SEV-2017-2018 and through the project MTM2017-82160-C2-1-P funded by (AEI/FEDER, UE)
and acronym ``HAQMEC'', and by 2017 Leonardo grant for Researchers and Cultural Creators, BBVA Foundation.
The Foundation accepts no responsibility for the opinions, statements and contents included in the project
and/or the results thereof, which are entirely the responsibility of the authors.}



\begin{document}

\title[Maximal estimates for spherical means]{Maximal estimates for \\
	a generalized spherical mean Radon transform \\ acting on radial functions}

\begin{abstract}
We study a generalized spherical means operator,
viz.\ generalized spherical mean Radon transform, acting on radial functions.
As the main results, we find conditions for the associated maximal operator and its local variant
to be bounded on power weighted Lebesgue spaces. This translates, in particular, into almost everywhere convergence
to radial initial data results for solutions to certain Cauchy problems for classical Euler-Poisson-Darboux and wave equations.
Moreover, our results shed some new light on the interesting and important question of optimality
of the yet known $L^p$ boundedness results for the maximal operator in the general non-radial case.
It appears that these could still be notably improved, as indicated by our conjecture of the ultimate
sharp result.
\end{abstract}

\maketitle

\section{Preliminaries and statement of results} \label{sec:prel}

This paper is a natural continuation of our recent research from \cite{CNR}.
We study a generalized spherical means operator acting on radial functions.
In \cite{CNR} we viewed this operator as a family of integral transforms $\{M_t^{\a,\b} : t > 0\}$
acting on profile functions on $\mathbb{R}_+$ and found fairly precise estimates of the associated
integral kernels $K^{\a,\b}_t(x,z)$. This enabled us to prove two-weight $L^p-L^q(L^r_t)$ estimates for
$f \mapsto M^{\a,\b}_t f$, with $1 \le p,q \le \infty$ and $1 \le r < \infty$.
In the present work we focus on the more subtle limiting case $r=\infty$ and restrict to $p=q$,
which in the above language of mixed norm estimates corresponds to weighted $L^p$-boundedness of the maximal operator
$f \mapsto \sup_{t>0} |M^{\a,\b}_t f|$. Obtaining such results requires a different, in fact more tricky, approach
from that used in \cite{CNR}. It is worth emphasizing that both our works, \cite{CNR} and this one, were
to large extent motivated by connections of the generalized spherical means with solutions to a number of
classical initial-value PDE problems being of physical and practical importance; see e.g.\ \cite[Section 7]{CNR}
and references given there.

Let $n \ge 2$ and consider the generalized spherical means transformation
$$
M^{\b}f(x,t) = \mathcal{F}^{-1} \big( m_{\b}(t|\cdot|) \mathcal{F}f \big)(x),
$$
where $\mathcal{F}$ is the Fourier transform in $\mathbb{R}^n$ and the radial multiplier is given via
$$
m_{\b}(s) = 2^{\b+n/2-1} \Gamma(\b+n/2) \frac{J_{\b+n/2-1}(s)}{s^{\b+n/2-1}}, \qquad s > 0,
$$
with $J_{\nu}$ denoting the Bessel function of the first kind and order $\nu$.
The parameter $\b$ can, in general, be a complex number excluding $\b = -n/2, -n/2-1, -n/2-2,\ldots$.
For $\b=0$ one recovers the classical spherical means
$$
M^0 f(x,t) = \int_{S^{n-1}} f(x-ty)\, d\sigma(y), \qquad (x,t) \in \mathbb{R}^n \times \mathbb{R}_+,
$$
where $d\sigma$ is the normalized uniform measure on the unit sphere $S^{n-1} \subset \mathbb{R}^n$.
Clearly, $M^0f(x,t)$ returns the mean value of $f$ on the sphere centered at $x$ and of radius $t$.

For the maximal operator $M^{\b}_*f = \sup_{t >0} |M^{\b}f(\cdot,t)|$ Stein \cite{Stein} proved the following.
\begin{thm}[{\cite{Stein}}] \label{thm:Stein}
Let $n \ge 3$. Then $M^{\b}_*$ is bounded on $L^p(\mathbb{R}^n)$ provided that
\begin{equation} \label{cndthm11}
1 < p \le 2 \;\; \textrm{and} \;\; \b > 1-n+\frac{n}p, \quad \textrm{or} \quad p > 2 \;\; \textrm{and} \;\; \b > \frac{2-n}p.
\end{equation}
\end{thm}

This result was enhanced in the sense of admitted parameters and dimensions by subsequent authors:
Bourgain \cite{Bourgain}, Mockenhoupt, Seeger and Sogge \cite{MSS} and recently by Miao,
Yang and Zheng \cite{MYZ}, see the historical comments in \cite[p.\,4272]{MYZ}.
All these refinements can be stated altogether as follows, cf.\ \cite[Theorem 1.1]{MYZ}.
\begin{thm}[{\cite{MYZ}}] \label{thm:MYZ}
Let $n \ge 2$. Then $M^{\b}_*$ is bounded on $L^p(\mathbb{R}^n)$ provided that
\begin{equation} \label{cndthm12}
2 < p \le \frac{2n+2}{n-1} \;\; \textrm{and} \;\; \b > \frac{1-n}{4}+\frac{3-n}{2p}, \quad \textrm{or} \quad
	p > \frac{2n+2}{n-1} \;\; \textrm{and} \;\; \b > \frac{1-n}p.
\end{equation}
\end{thm}
The range of $\b$ in Theorem \ref{thm:MYZ}, for $p > 2$, is strictly wider than in Theorem \ref{thm:Stein};
see Figure \ref{fig} below. However,
according to our best knowledge, it is not known whether it is already optimal.
We strongly believe it is not, see our Conjecture \ref{con:main} below.
We remark that both Theorems \ref{thm:Stein} and \ref{thm:MYZ} were originally proved for complex $\b$, but
for our purposes it is enough to state them for real values of the parameter.

A restriction of $M^{\b}$ to radially symmetric functions is still of interest and, moreover,
admits a more explicit finer analysis that potentially leads to more general or stronger theorems. In this connection we invoke
two quite recent results of Duoandikoetxea, Moyua and Oruetxebarria \cite{DMO,DMO2}. The first of them is a characterization
of weighted $L^p_{\textrm{rad}}$-boundedness of $M^0_*$ with radial power weights involved; here and elsewhere
the subscript ``rad'' indicates the subspace of radial functions.
\begin{thm}[{\cite{DMO}}] \label{thm:Duo1}
Let $n \ge 2$. The operator $M^0_*$ is bounded on $L^p_{\textrm{rad}}(\mathbb{R}^n,|x|^{\gamma}dx)$ for $1< p < \infty$
(for $2<p<\infty$ in case $n=2$) if and only if
$$
1-n \le \gamma < p(n-1)-n.
$$
\end{thm}
Another result in this spirit provides sufficient conditions for weighted $L^p_{\textrm{rad}}$-boundedness
of $M^{(3-n)/2}_*$, again with radial power weights involved.
\begin{thm}[{\cite{DMO2}}] \label{thm:Duo2}
Let $n \ge 2$ and $1 < p < \infty$. Then the operator $M^{(3-n)/2}_*$ is bounded on
$L^p_{\textrm{rad}}(\mathbb{R}^n,|x|^{\gamma}dx)$ provided that
$$
\frac{n-3}{2} p + 1 -n < \gamma < \frac{n+1}2 p -n,
$$
where in case $n=2$ the lower bound for $\gamma$ should be replaced by $-3/2$, and with the first inequality weakened for
odd dimensions $n$.
\end{thm}
Here and elsewhere by weakening a strict inequality we mean replacing ``$<$'' by ``$\le$''.
We take this opportunity to note that the statement of \cite[Theorem 1.1]{DMO2}
(an unweighted specification of Theorem \ref{thm:Duo2}) contains a small error, the upper
bound for $p$ in even dimensions higher than $3$ should be given by strict inequality, cf.\ \cite[Lemma 4.3]{DMO2}.
This problem affects the abstract of \cite{DMO2} as well. The radial improvement for the unweighted estimates
occurs only in odd dimensions higher than $3$.

We omit here discussion of more sophisticated mapping properties of $M^{\b}_*$, like boundedness from $L^1$ to weak $L^1$,
or from the Lorentz space $L^{1,1}$ to weak $L^1$ (which correspond to weak and restricted weak type estimates for
the maximal operator, respectively).
For this kind of results, especially in the radial case, we refer to \cite{DMO,DMO2} and references given there.

The maximal operator $M^{\a,\b}_*$ we shall study generalizes the restriction of $M^{\b}_*$ to radial functions,
since, in a sense, it covers a continuous range of dimensions $n=2\a+2$, $\a > -1$. Our main result,
Theorem \ref{thm:main} below, contains
Theorems \ref{thm:Duo1} and \ref{thm:Duo2} as special cases; in particular, we deliver a partly alternative
and seemingly simpler proof of Theorem \ref{thm:Duo2}. Moreover, our more general perspective sheds new light on
the discrepancy between $n=2$ and higher dimensions in Theorems \ref{thm:Duo1} and \ref{thm:Duo2}, as well as on
the discrepancy between odd and even dimensions in Theorem \ref{thm:Duo2}.
Finally, we gain some intuition that enables us to conjecture an optimal result in the spirit of Theorems \ref{thm:Stein}
and \ref{thm:MYZ}, see Conjecture \ref{con:main} below.

Let $\a > -1$ and $\a+\b > -1/2$.
For each $t>0$ consider the integral operator
$$
M_t^{\a,\b}f(x)=\int_0^{\infty}K^{\a,\b}_t(x,z)f(z)\, d\mu_{\a}(z), \quad x \in \mathbb{R}_+,
$$
with the measure $d\mu_{\a}(x)=x^{2\a+1}\,dx$ and the kernel given by
$$
K_t^{\a,\b}(x,z) = \frac{2^{\a+\b}\Gamma(\a+\b+1)}{t^{\a+\b}(xz)^{\a}} \int_0^{\infty} J_{\a+\b}(ty)
	J_{\a}(xy) J_{\a}(zy) y^{1-\a-\b}\, dy.
$$
This kernel is well defined for $(t,x,z) \in \mathbb{R}^3_+$ such that, in general, $t \neq |x-z|$ and $t \neq x+z$.
Note that the integral here converges absolutely when $\a+\b > 1/2$, but otherwise the convergence at $\infty$ is
only conditional, in the Riemann sense.

For a radial function $f = f_0(|\cdot|)$ in $L^2(\mathbb{R}^n)$, $n \ge 2$, $M^{\b}f(x,t)$ is for each $t >0$ a radial
function in $x \in \mathbb{R}^n$ whose profile is given by $M^{n/2-1,\b}_t f_0$; see \cite[Corollary 4.2]{CNR}.
Clearly, the maximal operators $M^{\b}_*$ and
$$
M_*^{\a,\b}f(x)=\sup_{t>0}\big|M_t^{\a,\b}f(x)\big|
$$
are connected in the same way. Thus $M^{\b}_*$ is bounded on $L^p_{\textrm{rad}}(\mathbb{R}^n,|x|^{\gamma}dx)$
if and only if $M^{n/2-1,\b}_*$ is bounded on $L^p(\mathbb{R}_+,x^{\gamma}d\mu_{n/2-1})$.

We now formulate the main result of this paper.
\begin{thm} \label{thm:main}
Assume that $\a > -1$ and $\a+\b > -1/2$. Let $1 < p < \infty$ and $\delta \in \mathbb{R}$.
Then the maximal operator $M^{\a,\b}_*$ is bounded on $L^p(\mathbb{R}_+,x^{\delta}dx)$ if
\begin{equation} \label{cnd0}
\frac{1}{p} < \a + \b + \frac{1}2
\end{equation}
and
\begin{itemize}
\item[(a)] in case $\b \ge 1$,
\begin{equation} \label{cnda}
-1 < \delta < (2\a+2)p-1,
\end{equation}
\item[(b)] in case $0 < \b < 1$,
\begin{equation} \label{cndb}
\frac{-\b}{[\b \vee (\a+\b+1/2)]\wedge 1} < \delta < (2\a+\b+1)p-1,
\end{equation}
with the first inequality weakened when $\a+\b > 1/2$,
\item[(c)] in case $\b \le 0$,
\begin{equation} \label{cndc}
-\b p < \delta < (2\a+\b+1)p-1,
\end{equation}
with the first inequality weakened when $-\b \in \mathbb{N}$ or $\a+\b > 1/2$.
\end{itemize}
\end{thm}

We presume that Theorem \ref{thm:main} is sharp, but at the moment we are not able to give a complete justification
of this statement. Nevertheless, we can show that Theorem \ref{thm:main} is optimal,
or very close to optimal, for most choices of the parameters $\a,\b$.
\begin{propo} \label{prop:sharp}
Assume that $\a > -1$ and $\a+\b > -1/2$. Let $1 < p < \infty$ and $\delta \in \mathbb{R}$.
If $M_{*}^{\a,\b}$ is bounded on $L^p(\mathbb{R}_+,x^{\delta}dx)$, then \eqref{cnd0} holds and, moreover,
$$
-1 < \delta, \qquad -\b p \le \delta, \qquad \delta < (2\a+2)p-1, \qquad \delta < (2\a+\b+1)p-1.
$$
\end{propo}
It follows that Theorem \ref{thm:main} gives sharp conditions when $\b \ge 1$ or when $\b \le 0$ and
[$-\b \in \mathbb{N}$ or $\a+\b > 1/2$] or when $\b \in (0,1)$ and $\a \le -1/2$.
For $\b \le 0$ and $-\b \notin \mathbb{N}$ and $\a+\b \le 1/2$ Theorem \ref{thm:main}
is sharp up to a possible weakening the strict inequality
in the lower bound for $\delta$ in \eqref{cndc}; this remains to be settled. In case $\b \in (0,1)$ and $\a > -1/2$,
the upper bound for $\delta$ in \eqref{cndb} is sharp, while optimality of the lower one remains an open question.
See also Remark \ref{rem:sharp} at the end.

Note that in Theorem \ref{thm:main} condition \eqref{cnd0} is meaningful only when $\a+\b < 1/2$.
Further, considering the lower bound in \eqref{cndb}, notice that
$$
\frac{-\b}{[\b \vee (\a+\b+1/2)]\wedge 1} =
\begin{cases}
-\b, & \a+\b \ge 1/2,\\
\frac{-\b}{\a+\b+1/2}, & \a+\b < 1/2 \;\; \textrm{and} \;\; \a> -1/2,\\
-1, & \a+\b< 1/2 \;\; \textrm{and} \;\; \a \le -1/2.
\end{cases}
$$
Observe also that \eqref{cnda} is in fact the $A_p$ condition for the power weight $x^{\delta-(2\a+1)}$
in the context of the space of homogeneous type $(\mathbb{R}_+,d\mu_{\a},|\cdot|)$.

In the circumstances of Theorem \ref{thm:main}, assuming in addition that $2\a+\b > -1$,
for each $p$ satisfying \eqref{cnd0} there is always a non-trivial interval of $\delta$ for which the
$L^p(x^{\delta}dx)$ boundedness holds. On the other hand, when $2\a+\b \le -1$ (i.e.\ $(\a,\b)$ is inside 
or on the rightmost side of the small
triangle with vertices $(-1/2,0)$, $(-1,1/2)$, $(-1,1)$) Theorem \ref{thm:main} gives no $L^p(x^{\delta}dx)$ boundedness
(the range of $\delta$ is empty for each $p$).

Taking $\delta= \gamma + 2\a+1$ and $\a=n/2-1$, $n \ge 2$ in Theorem \ref{thm:main}, and either $\b=0$ or $\b=(3-n)/2$,
one recovers Theorems \ref{thm:Duo1} and \ref{thm:Duo2}, respectively.
Furthermore, specifying Theorem \ref{thm:main} to the natural weight we set $\delta = 2\a+1$ and get the following.
\begin{coro} \label{cor:main}
Assume that $\a > -1$ and $\a+\b > -1/2$. Let $1 < p < \infty$. Then
$M^{\a,\b}_*$ is bounded on $L^p(\mathbb{R}_+,d\mu_{\a})$ provided that condition \eqref{cnd0} is satisfied and
\begin{equation} \label{cnd17}
\frac{-\b}{2\a+1} < \frac{1}{p} < 1-\frac{1-\b}{2\a+2},
\end{equation}
with the first inequality weakened if $-\b \in \mathbb{N}$ or $\a+\b > 1/2$.
These conditions are optimal up to a possible weakening of the strict inequality
in the lower bound in \eqref{cnd17} when $0 < -\b \notin \mathbb{N}$ and $\a+\b \le 1/2$.
\end{coro}

Observe that, in the context of Corollary \ref{cor:main}, condition \eqref{cnd17} is always satisfied in case $\b \ge 1$,
and for $0 \le \b < 1$ the lower bound in \eqref{cnd17} is always true.
Note also that there is no $p$ for which the boundedness holds if
either $\b \in (0,1)$ and $2\a+\b+1 \le 0$ or $\b<0$ and $\a + \b \le -1/4 - 1/(16\a + 12)$.
Otherwise, there is always
a non-trivial interval of $p$ for which $M^{\a,\b}_*$ is bounded on $L^p(d\mu_{\a})$.

Taking $\a = n/2-1$, $n \ge 2$, in Corollary \ref{cor:main} we obtain
\begin{coro} \label{cor:main2}
Let $n \ge 2$ and $1< p < \infty$. Then $M^{\b}_*$ is bounded on $L^p_{\textrm{rad}}(\mathbb{R}^n)$ provided that
\begin{equation} \label{cnd18}
p \le \frac{2n-1}{n-1} \;\; \textrm{and} \;\; \b > 1-n+n/p, \quad \textrm{or} \quad
	p > \frac{2n-1}{n-1} \;\; \textrm{and} \;\; \b > \frac{1-n}p,
\end{equation}
with the last inequality weakened when $-\b \in \mathbb{N}$ or $\b > \frac{3-n}2$.
This condition is optimal up to a possible weakening of the last inequality in \eqref{cnd18} in case
$-\b \notin \mathbb{N}$ and $\b \le \frac{3-n}2$.
\end{coro}

We believe that weakening the inequality in Corollary \ref{cor:main2} reflects radial improvement of mapping properties
of $M^{\b}_*$. On the other hand, the fact that for $2 < p < 2(n+1)/(n-1)$
condition \eqref{cnd18} is strictly less restrictive than condition \eqref{cndthm12}, see Figure \ref{fig} below,
most probably indicates non-optimality of the yet known results on $L^p$-boundedness of $M^{\b}_*$
stated in Theorem \ref{thm:MYZ}; cf.\ \cite[Section 3, Problem (1)]{MYZ}.

Taking into account Corollary \ref{cor:main2} and known sharp results on boundedness of $M^0_*$ on $L^p(\mathbb{R}^n)$
we state the following.
\begin{conj} \label{con:main}
Let $n \ge 2$ and $1< p < \infty$.
The operator $M^{\b}_*$ is bounded on $L^p(\mathbb{R}^n)$ if and only if condition \eqref{cnd18} is satisfied.
\end{conj}

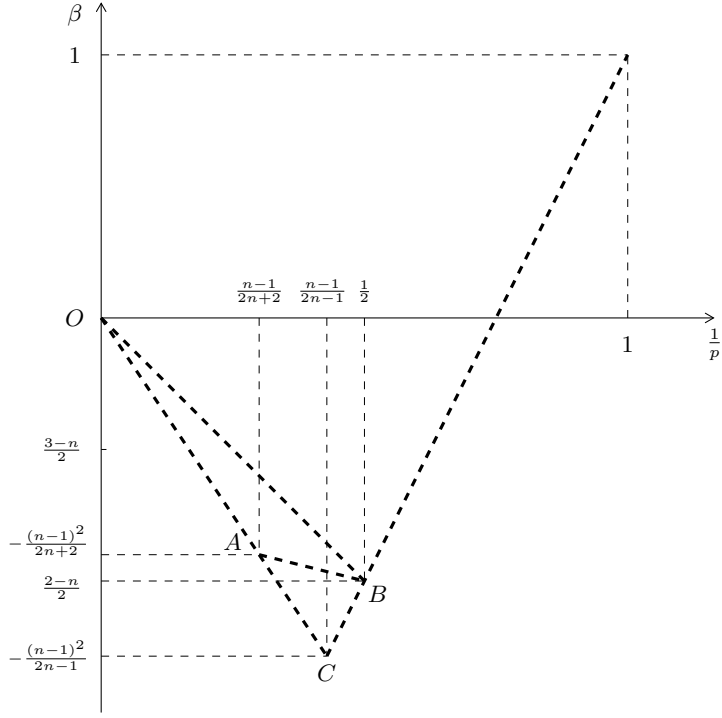
\begin{figure}
\centering
\begin{tikzpicture}[scale=3.5]
\small{
\draw[arrows=-angle 60] (0,-1.5) -- (0,1.2);
\draw[arrows=-angle 60] (0,0) -- (2.33,0);
\node at (-0.1,0) {$O$};
\node at (-0.1,1.15) {$ \beta$};
\node at (2.33,-0.1) {$\tfrac{1}{p}$};
\node at (-0.1,1) {$1$};
\node at (2,-0.1) {$1$};
}
\tiny{
\node at (-0.2,-9/7)  {$-\tfrac{(n-1)^2}{2n-1}$}; 
\node at (0.6,0.1) {$\tfrac{n-1}{2n+2}$}; 
\node at (1,0.1) {$\tfrac{1}{2}$}; 
\node at (-0.2,-0.85) {$-\tfrac{(n-1)^2}{2n+2}$}; 
\node at (-0.15,-1.025) {$\tfrac{2-n}{2}$}; 
\node at (0.84,0.1) {$\tfrac{n-1}{2n-1}$}; 
\node at (-0.15,-0.5) {$\tfrac{3-n}{2}$}; 
}
\small{
\node at (0.5,-0.85) {$A$};
\node at (6/7, -1.35) {$C$};
\node at (1.05,-1.05) {$B$};
\draw[very thin, dashed] (0,1) -- (2,1); 
\draw[very thin, dashed] (2,0) -- (2,1); 
\draw[very thin, dashed] (0,-9/7) -- (6/7,-9/7); 
\draw[very thin, dashed] (6/7,0) -- (6/7,-9/7); 
\draw[very thin, dashed] (0,-0.9) -- (0.6,-0.9); 
\draw[very thin, dashed] (0.6,0) -- (0.6,-0.9); 
\draw[very thin, dashed] (0,-1) -- (1,-1); 
\draw[very thin, dashed] (1,0) -- (1,-1); 
\draw[very thick, dashed] (0,0) -- (0.6,-0.9); 
\draw[very thick, dashed] (0.6,-0.9) -- (1,-1); 
\draw[very thick, dashed] (6/7,-9/7) -- (2,1); 
\draw[very thick, dashed] (0.6,-0.9) -- (6/7,-9/7); 
\draw[very thick, dashed] (0,0) -- (1,-1); 
\draw[very thin, dashed] (0,-0.5) -- (0.02,-0.5);
}
\end{tikzpicture} 

\caption{Regions $OAB$ and $ABC$ visualize differences between Stein's condition \eqref{cndthm11}, the less restrictive
condition \eqref{cndthm12} due to Miao et al., and the least restrictive condition \eqref{cnd18} from our radial case
and Conjecture \ref{con:main} for the general case, respectively.
Picture for $n=4$, with different axes scaling.} \label{fig}
\end{figure}

Finally, we comment on pointwise almost everywhere convergence $M^{\a,\b}_t f \to f$ as $t \to 0^+$.
This is an important issue, due to connections of $M_t^{\a,\b}$ with solutions to certain initial-value PDE problems.
It is well known that the key ingredient leading to such results are suitable mapping properties of $M^{\a,\b}_*$,
or actually of a smaller truncated maximal operator
$$
M^{\a,\b}_{*,\textrm{tru}}f(x) = \sup_{0 < t < x/2} \big|M^{\a,\b}_t f(x) \big|,
$$
see e.g.\ \cite[Section 5]{DMO2} and the proof of \cite[Corollary 3.6]{CoCoSt}.
Thus, from this point of view, a remarkable by-product
of our analysis related to $M^{\a,\b}_*$ is the following control in terms of the local Hardy-Littlewood maximal
operator $L$ (see Section \ref{ssec:23} below for the definition).
\begin{propo} \label{prop:Mloc}
Assume that $\a > -1$ and $\a+\b > -1/2$. Then 
\begin{equation} \label{Mlocctrl}
M^{\a,\b}_{*,\textrm{tru}}f(x) \lesssim
	\begin{cases}
		Lf(x), & \quad \textrm{if} \;\; \a+\b \ge 1/2,\\
		x^{\b}[L(z^{-\b}f)^{1+\theta}(x)]^{1/(1+\theta)}, & \quad \textrm{if} \;\; \a+\b < 1/2,
	\end{cases}
\end{equation}
uniformly in locally integrable functions $f$ on $\mathbb{R}_+$ and $x \in \mathbb{R}_+$,
with any fixed $\theta > 0$ such that $\frac{1}{1+\theta} < \a + \b +1/2$.

Consequently, if $1 < p < \infty$ satisfies condition \eqref{cnd0} then
$M^{\a,\b}_{*,\textrm{tru}}$ is bounded on $L^p(\mathbb{R}_+,x^{\delta}dx)$ for any $\delta \in \mathbb{R}$.
Moreover, in case $\a+\b \ge 1/2$, $M^{\a,\b}_{*,\textrm{tru}}$ is of weak type $(1,1)$ with respect to the measure
space $(\mathbb{R}_+,x^{\delta}dx)$ with any $\delta \in \mathbb{R}$.
\end{propo}

\begin{coro} \label{cor:ae}
Let $\a,\b$ and $p$ be as in Proposition \ref{prop:Mloc}. Then
$$
M^{\a,\b}_t f(x) \to f(x) \;\; \textrm{for}\;\; \textrm{a.a.}\;\; x \in \mathbb{R}_+ \;\; \textrm{as}\;\; t \to 0^+
$$
for any locally integrable function $f$ on $\mathbb{R}_+$ in case $\a+\b \ge 1/2$, or for any $f$ locally
in $L^p(\mathbb{R}_+,dx)$ with $p$ satisfying condition \eqref{cnd0} in case $\a+\b < 1/2$.
\end{coro}

Consequently, we obtain almost everywhere convergence to initial data results in contexts of differential problems whose
solutions express via $M^{\a,\b}_t$. In particular, this pertains to the following Cauchy initial-value problems,
see e.g.\ \cite[Section 7]{CNR}.
\begin{itemize}
\item[(i)] Euler-Poisson-Darboux equation in $\mathbb{R}^n$ with a locally integrable/locally in $L^p$ radial initial position
and null initial speed.
\item[(ii)] Wave equation in $\mathbb{R}^n$ with null initial position and a locally integrable
radial initial speed (this case was already treated before, see \cite[Section 5]{DMO2}).
\item[(iii)] Euler-Poisson-Darboux equation based on the one-dimensional Bessel operator with
a locally integrable/locally in $L^p$ initial position and null initial speed.
\item[(iv)] Wave equation based on the one-dimensional Bessel operator with null initial position and a locally
integrable/locally in $L^p$ initial speed.
\end{itemize}
In (i), (iii) and (iv) by ``locally integrable/locally in $L^p$'' we mean local integrability in case $\a+\b \ge 1/2$,
and being locally in $L^p$ for some $p$ satisfying condition \eqref{cnd0} in case $-1/2 < \a+\b < 1/2$.

\subsection*{Structure of the paper}
The rest of the paper is organized as follows. Section \ref{sec:prep} is a technical preparation for the proof
of Theorem \ref{thm:main}. It provides estimates of the integral kernel $K^{\a,\b}_t(x,z)$,
decomposition of the $(t,x,z)$ space into suitable regions and definitions of special auxiliary operators
together with description of their $L^p$-behavior. Section \ref{sec:proof} contains the proof of Theorem \ref{thm:main}.
The main part is preceded by establishing suitable control in terms of the special operators. At the end
of Section \ref{sec:proof} a brief justification of Proposition~\ref{prop:Mloc} is given.
Finally, in Section \ref{sec:counter} we construct suitable counterexamples in order to prove Proposition \ref{prop:sharp}.

\subsection*{Notation}
Throughout the paper we use a fairly standard notation.
Thus $\mathbb{N} = \{0,1,2,\ldots\}$ and $\mathbb{R}_+=(0,\infty)$.
The symbols ``$\vee$'' and ``$\wedge$'' mean the operations of taking maximum and minimum, respectively.
We write $X\lesssim Y$ to indicate that $X\leq CY$ with a positive constant $C$
independent of significant quantities. We shall write $X \simeq Y$ when simultaneously
$X \lesssim Y$ and $Y \lesssim X$.

For the sake of brevity, we often omit $\mathbb{R}_+$ when denoting $L^p$ spaces related to the measure spaces
$(\mathbb{R}_+,x^{\delta}dx)$ and $(\mathbb{R}_+,d\mu_{\a})$.
We write $L^p_{\textrm{rad}}(\ldots)$ for the subspace of $L^p(\ldots)$ consisting of radial functions.

\subsection*{Acknowledgment}
We thank the referee for a constructive criticism that improved the presentation and led us to
enrich the paper with the sharpness results.

\section{Technical preparation} \label{sec:prep}

Define the main regions
\begin{align*}
E & = \big\{(t,x,z) \in \RR_+^3 : |x-z| < t < x+z\big\}, \\
F & = \big\{(t,x,z) \in \RR_+^3 : x+z < t \big\}.
\end{align*}

Following the strategy used in \cite{DMO2}, see also \cite{DMO},
we shall estimate the kernel $K_t^{\a,\b}(x,z)$ and then split the bounds according to suitably defined
sub-regions of $E$ and $F$.
Then we will estimate the resulting maximal operators independently and get control in terms of 
a number of special operators whose mapping properties are essentially known.
Altogether, this will give an overall control of the maximal operator $M_*^{\a,\b}$.

In preparatory subsections \ref{ssec:21}--\ref{ssec:23} below we gather the kernel estimates, suitable decompositions
of $E$ and $F$ and definitions of the special operators, respectively.

\subsection{Kernel estimates} \label{ssec:21}

We first rephrase \cite[Theorem 3.3]{CNR} in a way more convenient for our present purposes; see also \cite[Section 2.3]{CNR}.
In particular, we neglect parts concerning sharpness of the bounds for certain $\a$ and $\b$. 

\begin{thm}[\cite{CNR}] \label{thm:kerest}
Assume that $\a > -1$ and $\a+\b > -1/2$. Let $t,x,z > 0$.
The kernel $K_t^{\a,\b}(x,z)$ vanishes when $t < |x-z|$, whereas in $E \cup F$ the following uniform estimates hold.
\begin{itemize}
\item[(1)] If $-\b \in \mathbb{N}$, then the kernel vanishes in $F$ and
$$
|K_t^{\a,\b}(x,z)| \lesssim 
				\frac{(xz)^{-2\a-\b}}{t^{2\a+2\b}}
				\Big( \big[ t^2-(x-z)^2\big] \big[(x+z)^2-t^2\big]\Big)^{\a+\b-1/2} \quad \textrm{in} \; E.
$$
\item[(2)]
If $\a+\b > 1/2$ and neither $-\b \in \mathbb{N}$ nor $2\a+\b=0$, then
$$
|K_t^{\a,\b}(x,z)| \lesssim
		\begin{cases}
			\frac{(xz)^{-\a-1/2}}{t^{2\a+2\b}} \big[t^2-(x-z)^2\big]^{\a+\b-1/2} & \textrm{in} \; E, \\
			\frac{1}{t^{2\a+2\b}} \big[t^2-(x-z)^2\big]^{\b-1} & \textrm{in} \; F.
		\end{cases}
$$
\item[(3)]
If $\a+\b > 1/2$ and $2\a+\b=0$, then
$$
|K_t^{\a,\b}(x,z)| \lesssim
	\frac{1}{t^{-2\a}} \Big( \big[ t^2-(x-z)^2 \big] \big| t^2-(x+z)^2\big| \Big)^{-\a-1/2} \quad \textrm{in} \; E \cup F.
$$
\item[(4)]
If $\a+\b=1/2$ and neither $-\b \in \mathbb{N}$ nor $\b=1$, then
$$
|K_t^{\a,\b}(x,z)| \lesssim
		\begin{cases}
			 \frac{(xz)^{-\a-1/2}}{t}
				\log\frac{8xz}{(x+z)^2-t^2} & \textrm{in} \; E, \\
			\frac{1}{t} \big[t^2-(x-z)^2\big]^{-\a-1/2} \log\Big( 2 \frac{t^2-(x-z)^2}{t^2-(x+z)^2}\Big) & \textrm{in} \; F.
		\end{cases}
$$
\item[(5)]
If $\a+\b=1/2$ and $\b=1$, then
$$
K_t^{\a,\b}(x,z) \simeq \frac{1}t \quad \textrm{in} \; E \cup F.
$$
\item[(6)]
If $\a + \b < 1/2$ and neither $-\b \in \mathbb{N}$ nor $\b=1$ nor $\a+1/2 \in \mathbb{N}$, then 
$$
|K_t^{\a,\b}(x,z)| \lesssim
		\begin{cases}
			 \frac{(xz)^{-2\a-\b}}{t^{2\a+2\b}}
				\Big( \big[ t^2-(x-z)^2\big] \big[(x+z)^2-t^2\big]\Big)^{\a+\b-1/2} & \textrm{in} \; E, \\
			\frac{1}{t^{2\a+2\b}} \big[t^2-(x-z)^2\big]^{-\a-1/2} \big[t^2-(x+z)^2\big]^{\a+\b-1/2} & \textrm{in} \; F.
		\end{cases}
$$
\item[(7)]
If $\a + \b < 1/2$ and $\b=1$, then
$$
|K_t^{\a,\b}(x,z)| \lesssim
			\begin{cases}
			 \frac{(xz)^{-2\a-1}}{t^{2\a+2}}
				\Big( \big[ t^2-(x-z)^2\big] \big[(x+z)^2-t^2\big]\Big)^{\a+1/2} & \textrm{in} \; E, \\
			\frac{1}{t^{2\a+2}} & \textrm{in} \; F.
		\end{cases}
$$
\item[(8)]
If $\a+\b < 1/2$ and $\a+1/2 \in \mathbb{N}$, then
$$
|K_t^{\a,\b}(x,z)| \lesssim
		\begin{cases}
			\frac{(xz)^{-\a-1/2}}{t^{2\a+2\b}} \big[t^2-(x-z)^2\big]^{\a+\b-1/2} & \textrm{in} \; E, \\
			\frac{1}{t^{2\a+2\b}} \big[t^2-(x-z)^2\big]^{-\a-1/2} \big[t^2-(x+z)^2\big]^{\a+\b-1/2} & \textrm{in} \; F.
		\end{cases}
$$
\end{itemize}
\end{thm}

The above theorem distinguishes cases, or rather segments and lines in the $(\a,\b)$ plane, where the estimates
are better comparing to other neighboring $(\a,\b)$. However, this does not seem to be reflected in power weighted
$L^p$, $1<p<\infty$, boundedness of $M^{\a,\b}_*$. This claim is based on our detailed analysis of the maximal operator
in each of the cases related to items (1)--(8) of Theorem \ref{thm:kerest}. Even though we could often obtain seemingly
better control of $K^{\a,\b}_*$ in terms of the auxiliary special operators, it led us to the same mapping properties
as stated in Theorem \ref{thm:main}. Therefore, we simplify things already at this stage and derive less accurate kernel bounds,
but with simpler structure (in particular, containing no logarithmic expressions), that will be sufficient for our purpose.
However, this strategy might not be appropriate for studying more subtle mapping properties of $M^{\a,\b}_*$, like
weak or restricted weak type estimates.

For notational convenience, define auxiliary kernels
\begin{align*}
\Phi_t^{\a,\b}(x,z) & = \chi_{E}(t,x,z) \frac{(xz)^{-\a-1/2}}{t^{2\a+2\b}} \big[t^2-(x-z)^2\big]^{\a+\b-1/2} \\
							& \qquad + \chi_{F}(t,x,z) \frac{1}{t^{2\a+2\b}} \big[t^2-(x-z)^2\big]^{\b-1},\\
\Psi_t^{\a,\b}(x,z) & = \chi_{E}(t,x,z) \frac{(xz)^{-2\a-\b}}{t^{2\a+2\b}}
				\Big( \big[ t^2-(x-z)^2\big] \big[(x+z)^2-t^2\big]\Big)^{\a+\b-1/2}\\
				& \qquad + \chi_{F}(t,x,z) \frac{1}{t^{2\a+2\b}} \big[t^2-(x-z)^2\big]^{-\a-1/2}
			\big[t^2-(x+z)^2\big]^{\a+\b-1/2}.
\end{align*}

\begin{thm} \label{thm:kermaj}
Assume that $\a > -1$ and $\a+\b > -1/2$. Fix an arbitrary $\varepsilon > 0$. The following estimates hold uniformly
in $(t,x,z) \in E \cup F$:
\begin{align*}
\big|K_t^{\a,\b}(x,z)\big| \lesssim
\begin{cases}
	\Phi_t^{\a,\b}(x,z), & \textrm{if} \;\; \a+\b > 1/2,\\
	\Psi_t^{\a,\b-\varepsilon}(x,z), & \textrm{if} \;\; \a+\b = 1/2,\\
	\Psi_t^{\a,\b}(x,z), & \textrm{if} \;\; \a+\b < 1/2,\\
	\Psi_t^{\a,\b}(x,z)\chi_{E}(t,x,z), & \textrm{if}\;\; -\b \in \mathbb{N}.
\end{cases}
\end{align*}
Moreover, $\chi_{E}(t,x,z) \Psi_t^{\a,\b}(x,z) \lesssim \chi_{E}(t,x,z) \Phi_t^{\a,\b}(x,z)$ when $\a + \b \ge 1/2$.
\end{thm}

\begin{proof}
All the asserted bounds are straightforward consequences of Theorem \ref{thm:kerest}.
Clearly,
$$
|K_t^{\a,\b}(x,z)| \lesssim \chi_{E}(t,x,z) \Psi_t^{\a,\b}(x,z) \quad \textrm{if} \;\; -\b \in \mathbb{N},
$$
for $(t,x,z) \in E \cup F$, is just Theorem \ref{thm:kerest}(1). We can write
$$
\chi_{E}(t,x,z) \Psi_t^{\a,\b}(x,z) = \chi_{E}(t,x,z) \Phi_t^{\a,\b}(x,z) \bigg[ \frac{(x+z)^2-t^2}{xz}\bigg]^{\a+\b-1/2}
$$
and, since the expression in square brackets is bounded from above by a constant, we infer that
\begin{equation} \label{maj1}
\chi_{E}(t,x,z)\Psi_t^{\a,\b}(x,z) \lesssim \chi_E(t,x,z) \Phi_t^{\a,\b}(x,z) \quad \textrm{when}\;\; \a+\b \ge 1/2.
\end{equation}

Next, assume that $\a+\b > 1/2$ and $2\a+\b=0$ (see Theorem \ref{thm:kerest}(3)), which forces $\a < -1/2$.
For $(t,x,z) \in E \cup F$ we have
\begin{align*}
 \frac{1}{t^{-2\a}} \Big( \big[ t^2-(x-z)^2 \big] \big| t^2-(x+z)^2\big| \Big)^{-\a-1/2}
& = \chi_E(t,x,z) \Phi_t^{\a,\b}(x,z) \bigg[ \frac{xz}{(x+z)^2-t^2}\bigg]^{\a+1/2} \\
	& \quad + \chi_F(t,x,z) \Phi_t^{\a,\b}(x,z) \bigg[ \frac{t^2-(x-z)^2}{t^2-(x+z)^2}\bigg]^{\a+1/2}.
\end{align*}
Here the expressions in square brackets are bounded from below by a positive constant. Hence, in view
of Theorem \ref{thm:kerest}(1)--(3) and \eqref{maj1}, we see that
$$
|K_t^{\a,\b}(x,z)| \lesssim \Phi_t^{\a,\b}(x,z) \quad \textrm{if} \;\; \a+\b > 1/2.
$$

Let now $\a+\b=1/2$ and fix $\varepsilon > 0$. For $(t,x,z) \in E$ we have (see Theorem \ref{thm:kerest}(4))
\begin{align*}
& \frac{(xz)^{-\a-1/2}}{t} \log\frac{8xz}{(x+z)^2-t^2} \\
& \quad \lesssim \frac{(xz)^{-\a-1/2}}{t} \bigg[ \frac{(x+z)^2-t^2}{xz}\bigg]^{\a+\b-\varepsilon-1/2} \\
& \quad = \Psi_t^{\a,\b-\varepsilon}(x,z) \bigg[ \frac{t^2}{t^2-(x-z)^2}\bigg]^{\a+\b-\varepsilon-1/2}
	\le \Psi_t^{\a,\b-\varepsilon}(x,z),
\end{align*}
where we used the fact that $\log y$ grows slower than the positive power $y^{-(\a+\b-\varepsilon-1/2)}$
(here $y$ positive and separated from zero). Further, for $(t,x,z) \in F$,
\begin{align*}
& \frac{1}t \big[ t^2-(x-z)^2\big]^{-\a-1/2} \log\bigg( 2\frac{t^2-(x-z)^2}{t^2-(x+z)^2}\bigg) \\
& \quad \lesssim \frac{1}t \big[ t^2-(x-z)^2\big]^{-\a-1/2}
	\bigg[\frac{t^2-(x+z)^2}{t^2-(x-z)^2}\bigg]^{\a+\b-\varepsilon-1/2} \\
& \quad = \Psi_t^{\a,\b-\varepsilon}(x,z) \bigg[ \frac{t^2}{t^2-(x-z)^2}\bigg]^{\a+\b-\varepsilon-1/2}
	\le \Psi_t^{\a,\b-\varepsilon}(x,z).
\end{align*}
This shows that (see Theorem \ref{thm:kerest}(4),(5))
$$
|K_t^{\a,\b}(x,z)| \lesssim \Psi_t^{\a,\b-\varepsilon}(x,z) \quad
	\textrm{if} \;\; \a+\b =1/2
$$
for $(t,x,z) \in E\cup F$
(to be precise, the case coming from Theorem \ref{thm:kerest}(5), i.e.\ $(\a,\b)=(-1/2,1)$,
is trivially included with no logarithms involved).

Finally, we consider $\a+\b < 1/2$, see Theorem \ref{thm:kerest}(6)--(8). If $\b=1$ then automatically $\a<-1/2$ and
for $(t,x,z) \in F$ we have (see Theorem \ref{thm:kerest}(7))
$$
\frac{1}{t^{2\a+2}} = \Psi_t^{\a,\b}(x,z) \bigg[ \frac{t^2-(x-z)^2}{t^2-(x+z)^2}\bigg]^{\a+1/2}
	\le \Psi_t^{\a,\b}(x,z).
$$
If $\a+1/2 \in \mathbb{N}$ then for $(t,x,z) \in E$ (see Theorem \ref{thm:kerest}(8))
$$
\frac{(xz)^{-\a-1/2}}{t^{2\a+2\b}} \big[ t^2-(x-z)^2 \big]^{\a+\b-1/2} =
	\Psi_t^{\a,\b}(x,z) \bigg[ \frac{(x+z)^2-t^2}{xz}\bigg]^{-(\a+\b-1/2)}.
$$
Altogether, this implies for $(t,x,z) \in E \cup F$
$$
|K_t^{\a,\b}(x,z)| \lesssim \Psi_t^{\a,\b}(x,z) \quad \textrm{if} \;\; \a+\b < 1/2.
$$

The proof of Theorem \ref{thm:kermaj} is complete.
\end{proof}

\subsection{Decompositions of $E$ and $F$} \label{ssec:22}

Inspired by the analysis presented in \cite{DMO2}, we define the following sub-regions of $E$ and $F$,
see Figure \ref{fig_reg} below,
\begin{align*}
E_1 = & E \cap \big\{(t,x,z) \in \RR_+^3 : t < x/2 \big\}, \\
E_2 = & E \cap \big\{(t,x,z) \in \RR_+^3 : x/2 \le t < 3x \big\}, \\
E_3 = & E \cap \big\{(t,x,z) \in \RR_+^3 : 3x \le t \big\}
\end{align*}
and
\begin{align*}
F_1 = & F \cap \big\{(t,x,z) \in \RR_+^3 : x < t < 3x \big\}, \\
F_2 = & F \cap \big\{(t,x,z) \in \RR_+^3 : 3x \le t \big\}, \\
F'_2 & = F_2 \cap \big\{(t,x,z) \in \RR_+^3 : z < (t-x)/2 \big\}, \\
F''_2 & = F_2 \cap \big\{(t,x,z) \in \RR_+^3 : (t-x)/2 \le z < t-x \big\}.
\end{align*}
Clearly, $E=E_1\cup E_2 \cup E_3$, $F = F_1 \cup F_2$, $F_2 = F'_2 \cup F''_2$, all the sums being disjoint.

\begin{figure}
\centering
\begin{tikzpicture}[scale=1.5]
\draw[arrows=-angle 60] (0,0) -- (0,5.5);
\draw[arrows=-angle 60] (0,0) -- (5.5,0);
\node at (-0.15,-0.18) {$0$};
\node at (0.12,5.4) {$t$};
\node at (5.4,0.13) {$z$};
\draw[very thin] (0.5,-0.05) -- (0.5,0.05);
\node at (0.5,-0.2) {$x/2$};
\draw[very thin] (1,-0.05) -- (1,0.05);
\node at (1,-0.2) {$x$};
\draw[very thin] (1.5,-0.05) -- (1.5,0.05);
\node at (1.5,-0.2) {$3x/2$};
\draw[very thin] (2,-0.05) -- (2,0.05);
\node at (2,-0.2) {$2x$};
\draw[very thin] (3,-0.05) -- (3,0.05);
\node at (3,-0.2) {$3x$};
\draw[very thin] (4,-0.05) -- (4,0.05);
\node at (4,-0.2) {$4x$};
\draw[very thin] (5,-0.05) -- (5,0.05);
\node at (5,-0.2) {$5x$};
\draw[very thin] (-0.05,0.5) -- (0.05,0.5);
\node at (-0.2,0.5) {$\frac{x}2$};
\draw[very thin] (-0.05,1) -- (0.05,1);
\node at (-0.2,1) {$x$};
\draw[very thin] (-0.05,2) -- (0.05,2);
\node at (-0.2,2) {$2x$};
\draw[very thin] (-0.05,3) -- (0.05,3);
\node at (-0.2,3) {$3x$};
\draw[very thin] (-0.05,4) -- (0.05,4);
\node at (-0.2,4) {$4x$};
\draw[very thin] (-0.05,5) -- (0.05,5);
\node at (-0.2,5) {$5x$};
\fill[black!5!white] (0,1) -- (1,0) -- (5,4) -- (5,5) -- (4,5) -- cycle;
\fill[black!20!white] (0,1) -- (4,5) -- (0,5) -- cycle;
\draw[thick] (0,1) -- (4,5);
\draw[thick] (1,0) -- (5,4);
\draw[thick] (0,1) -- (1,0);
\draw[thick] (0,3) -- (4,3);
\draw[thick] (1,3) -- (2,5);
\draw[thick] (0.5,0.5) -- (1.5,0.5);
\node at (0.7,4) {$\mathbf{F_2'}$};
\node at (2.2,4) {$\mathbf{F_2''}$};
\node at (4.1,4) {$\mathbf{E_3}$};  
\node at (0.7,2.4) {$\mathbf{F_1}$}; 
\node at (1.7,1.6) {$\mathbf{E_2}$}; 
\node at (1.03,0.33) {$\mathbf{E_1}$};
\draw[very thin, dashed] (0,0.5) -- (0.5,0.5) -- (0.5,0);
\draw[very thin, dashed] (1,0) -- (1,3);
\draw[very thin, dashed] (1.5,0) -- (1.5,0.5);
\draw[very thin, dashed] (2,0) -- (2,3);
\draw[very thin, dashed] (4,0) -- (4,3);
\node[rotate=-45] at (0.4,0.4) {$z=x-t$};
\node[rotate=45] at (4.5,3.3) {$z=t+x$};
\node[rotate=45] at (3.5,4.3) {$z=t-x$};
\node[rotate=63.43] at (1.4,4.3) {$z=\frac{t-x}{2}$}; 
\draw[arrows=-angle 60] (0,0) -- (0,5.5);
\end{tikzpicture} 

\caption{Sections of regions $E_i$ and $F_i$ given $x>0$ fixed.} \label{fig_reg}
\end{figure}
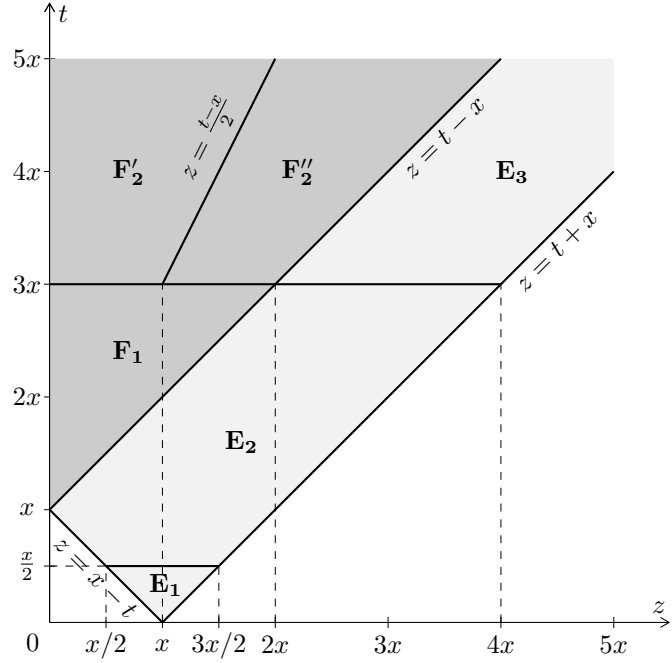

In what follows we will use uniform estimates of expressions in $t,x,z$ that hold specifically in one (or more) of the above
regions, cf.\ \cite{DMO2}.
We list them below for an easy further reference. Verification of these relations is straightforward and left
to the reader.
\begin{align}
\textrm{in}\; E_1: \qquad & x \simeq z, \label{relE1}\\
\textrm{in}\; E_2: \qquad & t^2-(x-z)^2 \lesssim tz, \label{relE} \\ 
\textrm{in}\; E_3: \qquad & t^2-(x-z)^2 \lesssim tx, \quad z \simeq t, \label{relE4}\\
\textrm{in}\; F_1: \qquad & t^2-(x-z)^2 \simeq x(t-x) \simeq x(t-x+z), \quad t^2-(x+z)^2 \simeq x(t-x-z), \label{relF1}\\
\textrm{in}\; F_2: \qquad & t^2-(x-z)^2 \simeq t(t+x-z), \quad t^2-(x+z)^2 \simeq t(t-x-z), \label{relF2}\\
\textrm{in}\; F'_2: \qquad & t^2-(x\pm z)^2 \simeq t^2, \label{relF2prim}\\
\textrm{in}\; F''_2: \qquad & z \simeq t \simeq t-x \simeq t - (x-z), \quad t+x-z \simeq t-z.\label{relF2bis}
\end{align}

\subsection{Special operators} \label{ssec:23}

We will use the following auxiliary operators whose mapping properties are essentially
known, see \cite{DMO,DMO2} and also references given there.

\noindent \textbf{A variant of the Hardy-Littlewood maximal operator} (denoted by $E_2$ in \cite{DMO})
$$
Df(x) = \sup_{0 \le a < x < b} \frac{1}{b^2-a^2} \int_a^b z |f(z)| dz.
$$
The operator $D$ is bounded on $L^p(x^{\gamma}dx)$, $1<p<\infty$,
if and only if $-1 < \gamma < 2p-1$.

\noindent \textbf{The local Hardy--Littlewood maximal operator}
$$
Lf(x)=\sup_{t<x/2}\frac{1}{2t}\int_{x-t}^{x+t}|f(z)|\,dz.
$$
The operator $L$ is bounded on $L^p(x^{\gamma}dx)$, $1<p<\infty$, for any $\gamma\in \RR$.

\noindent \textbf{The Hardy type operator}
$$
H_{\eta}f(x)=\frac{1}{x^{\eta}}\int_{0}^{x}z^{\eta-1}f(z)\,dz.
$$
The operator $H_{\eta}$ is bounded on $L^p(x^{\gamma}dx)$, $1<p<\infty$, if and only if $\gamma<\eta p-1$.

\noindent \textbf{The remote maximal operator}
$$
Rf(x)=\sup_{t>2x}\frac{1}{2x}\int_{t-x}^{t+x}|f(z)|\,dz.
$$
The operator $R$ is bounded on $L^p(x^{\gamma}dx)$, $1<p<\infty$, if and only if $\gamma\ge0$.

\noindent \textbf{The maximal operator $\mathbf{N}$}. For positive $\eta$, let
$$
N_{\eta}f(x)=\sup_{t>x}\frac{1}{t^{\eta}}\int_{0}^{t}z^{\eta-1}|f(z)|\,dz.
$$
The operator $N_{\eta}$ is bounded on $L^p(x^{\gamma}dx)$, $1<p<\infty$, if  $-1<\gamma<\eta p-1$.
At this point it is perhaps in order to note that the range given in \cite[p.\,1545]{DMO2} for the boundedness
of $N_{\eta}$ is wrong; consequently some statements in the proof of
\cite[Theorem 4.3]{DMO2} are not correct, nevertheless the result claimed there is true.

\noindent \textbf{The maximal operator $\mathbf{T}$}. For \textit{any} $\eta \in \RR$, consider the operator
\begin{equation}
\label{eq:Tb}
T_\eta f(x)=\sup_{t>2x}\int_{t/2}^{t}\frac{z^{\eta-1}|f(z)|}{(t-z+x)^{\eta}}\,dz.
\end{equation}
Boundedness of $T_{\eta}$ on power weighted $L^p$ spaces is shown in the following lemma, which is a simple extension
of \cite[Lemma 2.1]{DMO2}. Actually, this is even more than we need, we state item (b) only for the sake of completeness.

\begin{lema} \label{lem:T}
Let $1<p<\infty$. 
\begin{enumerate}
\item[(a)] For $\eta>1$, $T_{\eta}$ is bounded on $L^p(x^{\gamma}dx)$ for $\gamma\ge (\eta-1)p$.
\item[(b)] For $\eta = 1$, $T_{\eta}$ is bounded on $L^p(x^{\gamma}dx)$ for $\gamma > 0$.
\item[(c)] For $0 < \eta < 1$, $T_{\eta}$ is bounded on $L^p(x^{\gamma}dx)$ for $\gamma\ge \eta-1$.
\item[(d)] For $\eta\le 0$, $T_{\eta}$ is bounded on $L^p(x^{\gamma}dx)$ for $\gamma>-1$.
\end{enumerate} 
\end{lema}
\begin{proof}
Items (a) and (c) are proved in \cite{DMO2}, while (b) is commented there.

To show (d), assume that $\eta \le 0$ and observe that
$$
T_{\eta}f(x) = \sup_{t > 2x} \int_{t/2}^t \bigg( \frac{t-z+x}{z}\bigg)^{-\eta} \frac{|f(z)|}{z}\, dz
	\lesssim \sup_{t > 2x} \int_{t/2}^{t} \frac{|f(z)|}{z}\, dz \le \int_x^{\infty} \frac{|f(z)|}{z}\, dz.
$$
Thus $T_{\eta}$ is controlled in terms of the dual Hardy operator $f(x) \mapsto \int_x^{\infty} \frac{f(z)}{z}dz$, which is
well known to be bounded on $L^p(x^{\gamma}dx)$ if and only if $\gamma > -1$; see e.g.\ \cite[Lemma 6.3]{CNR}.
\end{proof}

\section{Proof of the main theorem} \label{sec:proof}

From now on we will consider only non-negative functions $f$. This will be enough, since our approach is based
on absolute estimates of the kernel.

Let $G \subset \RR_+^3$. We denote
\begin{align*}
\Phi^{\a,\b}_{*,G}f(x) & = \sup_{t >0} \int_0^{\infty} \chi_G(t,x,z) \Phi_t^{\a,\b}(x,z) f(z)\, d\mu_{\a}(z), \\
\Phi^{\a,\b}_* f(x) & = \Phi^{\a,\b}_{*,\mathbb{R}^3_+}f(x)
\end{align*}
and analogously for the $\Psi$ counterparts. Clearly, $\Phi^{\a,\b}_* f \le \Phi^{\a,\b}_{*,E}f + \Phi^{\a,\b}_{*,F}f$ and
similarly in case of the $\Psi$ operators.

\subsection{Control of $\Phi^{\a,\b}_*$} \label{ssec:31}
In this subsection we will establish a control of $\Phi^{\a,\b}_*$ in terms of the special operators defined in
Section \ref{ssec:23}. This will be done under the assumption $\a+\b \ge 1/2$.
\begin{lema} \label{lem:ctrlPhi}
Let $\a > -1$ and $\a+\b \ge 1/2$. Then
\begin{align*}
\Phi^{\a,\b}_{*,E}f(x) & \lesssim Lf(x) + H_{2\a+\b+1}f(4x) + x^{\b} R\big(z^{-\b}f\big)(x), \\
\Phi^{\a,\b}_{*,F}f(x) & \lesssim \chi_{\{\b \ge 1\}} H_{2\a+2}f(2x) + \chi_{\{\b < 1\}} H_{2\a+\b+1}f(2x)
					+ N_{2\a+2}f(x) + T_{1-\b}f(x),
\end{align*}
uniformly in $f \ge 0$ and $x > 0$.
\end{lema}

\begin{proof}
We will bound the maximal operators related to restrictions of $\Phi^{\a,\b}_t(x,z)$ to the sub-regions of $E$ and $F$
specified in Section \ref{ssec:22}.

\noindent{\textbf{Region $E_1$.}}
In view of \eqref{relE1}, taking into account that $\a+\b-1/2\ge 0$, we have
\begin{align*}
\Phi^{\a,\b}_{*,E_1}f(x) & \lesssim
\sup_{t<x/2}\frac{x^{-\a-1/2}}{t^{2\a+2\b}}\int_{x-t}^{x+t}\big[t^2-(x-z)^2\big]^{\a+\b-1/2}f(z)z^{\a+1/2} \,dz \\
& \lesssim \sup_{t<x/2}\frac{1}{t}\int_{x-t}^{x+t}f(z)\,dz=2Lf(x).
\end{align*}

\noindent{\textbf{Region $E_2$.}}
In this region $t \simeq x$, so using \eqref{relE} and $\a+\b-1/2 \ge 0$, we get
$$
\Phi^{\a,\b}_{*,E_2} f(x) \lesssim \sup_{x/2 \le t < 3x} \frac{1}{x^{2\a+\b+1}} \int_{|x-t|}^{x+t} z^{2\a+\b}f(z)\, dz
	\le H_{2\a+\b+1}f(4x).
$$

\noindent{\textbf{Region $E_3$.}}
Using \eqref{relE4} and $\a+\b-1/2\ge 0$ gives
\begin{align*}
\Phi^{\a,\b}_{*,E_3}f(x) &
 \lesssim \sup_{t \ge 3x} x^{\b} \frac{1}{x} \int_{t-x}^{t+x} z^{-\a-\b-1/2} f(z) z^{\a+1/2}\, dz
\lesssim x^{\b} R\big(z^{-\b}f\big)(x).	
\end{align*}

\noindent{\textbf{Region $F_1$.}}
Now $t \simeq x$. Using the fact that $t+x-z \simeq x$ in $F_1$, see \eqref{relF1}, we can write
$$
\Phi^{\a,\b}_{*,F_1}f(x) \lesssim \sup_{x < t < 3x} \frac{1}{x^{2\a+\b+1}} \int_0^{t-x} (t-x+z)^{\b-1} f(z) z^{2\a+1}\, dz.
$$
Here the factor $(t-x+z)^{\b-1}$ is controlled by $(t-x)^{\b-1} \lesssim x^{\b-1}$ when $\b \ge 1$, and by $z^{\b-1}$
otherwise. This leads directly to the bound
$$
\Phi^{\a,\b}_{*,F_1}f(x) \lesssim \chi_{\{\b \ge 1\}} H_{2\a+2}f(2x) + \chi_{\{\b < 1\}} H_{2\a+\b+1}f(2x).
$$

\noindent{\textbf{Region $F'_2$.}}
In view of \eqref{relF2prim},
$$
\Phi^{\a,\b}_{*,F'_2}f(x) \lesssim \sup_{t \ge 3x} \frac{1}{t^{2\a+2}} \int_0^{(t-x)/2} f(z) z^{2\a+1}\, dz
	\le N_{2\a+2}f(3x) \le N_{2\a+2}f(x).
$$

\noindent{\textbf{Region $F''_2$.}}
Here we use \eqref{relF2bis} and \eqref{relF2} obtaining
\begin{align*}
\Phi^{\a,\b}_{*,F''_2}f(x) & \lesssim \sup_{t \ge 3x} \int_{(t-x)/2}^{t-x} \frac{z^{-\b} f(z)}{(t-z)^{1-\b}}\, dz 
= \sup_{t \ge 2x} \int_{t/2}^{t} \frac{z^{-\b}f(z)}{(t+x-z)^{1-\b}}\, dz = T_{1-\b}f(x).
\end{align*}

The bounds of Lemma \ref{lem:ctrlPhi} follow.
\end{proof}

\subsection{Control of $\Psi^{\a,\b}_{*,E}$} \label{ssec:32}
We will prove the following.
\begin{lema} \label{lem:ctrlPsiE}
Let $\a > -1$ and $-1/2 < \a+\b < 1/2$. Then, given any $\theta > 0$ such that $\frac{1}{1+\theta} < \a+\b+1/2$,
$$
\Psi^{\a,\b}_{*,E}f(x) \lesssim
	x^{\b}\big[ D(z^{-\beta}f)^{1+\theta}(x)\big]^{\frac{1}{1+\theta}} 
	+ x^{\b}\big[ R(z^{-\beta}f)^{1+\theta}(x)\big]^{\frac{1}{1+\theta}},
$$
uniformly in $f \ge 0$ and $x > 0$.
\end{lema}

\begin{proof}
We first reduce the task to the special case $\b=0$. Observe that
$$
\Psi^{\a,\b}_{*,E}f(x) = \sup_{t > 0} \frac{x^{\b}}{(tx)^{2\a+2\b}} \int_{|x-t|}^{x+t}
  \Big( \big[ t^2-(x-z)^2\big] \big[(x+z)^2-t^2\big]\Big)^{\a+\b-1/2} z^{-\b} f(z) z\, dz.
$$
Thus
$$
	\Psi^{\a,\b}_{*,E}f(x) = x^{\b} \Psi^{\a+\b,0}_{*,E}\big(z^{-\b}f\big)(x)
$$
and hence it is enough to show that, given $-1/2 < \a < 1/2$,
\begin{equation} \label{redPsiE}
\Psi^{\a,0}_{*,E}f(x) \lesssim
	\big[ Df^{1+\theta}(x)\big]^{\frac{1}{1+\theta}} + \big[ Rf^{1+\theta}(x)\big]^{\frac{1}{1+\theta}}
\end{equation}
with any fixed $\theta > 0$ such that $\frac{1}{1+\theta} < \a+1/2$.

The proof of \eqref{redPsiE} is a straightforward generalization of the reasoning from \cite{DMO}, see the proof
of \cite[Theorem 3.1(b)]{DMO}. We present briefly some details for the reader's convenience.
Using the identity
\begin{equation} \label{id_tr}
\big[ t^2-(x-z)^2 \big] \big[(x+z)^2-t^2\big]=\big[(t+x)^2-z^2\big]\big[z^2-(x-t)^2\big]
\end{equation}
and then H\"older's inequality one gets the bound
\begin{align*}
\Psi^{\a,0}_{*,E}f(x) & \le \sup_{t>0} \Bigg[
\frac{1}{(xt)^{2\a}} \bigg( \int_{|t-x|}^{t+x} z f^{1+\theta}(z)\, dz \bigg)^{\frac{1}{1+\theta}} \\ & \qquad \times
		\bigg( \int_{|t-x|}^{t+x} \Big( \big[(t+x)^2-z^2\big]\big[z^2-(x-t)^2\big] \Big)^{(2\a-1)\frac{1+\theta}{2\theta}}
		z\, dz \bigg)^{\frac{\theta}{1+\theta}}\Bigg].
\end{align*}
The second integral here converges when $(2\a-1)\frac{1+\theta}{2\theta} > -1$, i.e.\ when $\frac{1}{1+\theta}< \a+1/2$.
In such the case it is comparable to $(xt)^{(2\a-1)\frac{1+\theta}{\theta}+1}$, see Lemma \ref{lem:iest}
in Section \ref{sec:counter} below.
Consequently, we arrive at the bound
$$
\Psi^{\a,0}_{*,E}f(x) \lesssim
	\sup_{t>0}\bigg( \frac{1}{xt} \int_{|t-x|}^{t+x} z f^{1+\theta}(z)\, dz \bigg)^{\frac{1}{1+\theta}}.
$$

From here one proceeds by splitting the supremum into $t \le 2x$ and $t> 2x$
(which corresponds essentially to estimating separately $\Psi^{\a,0}_{*,E_1\cup E_2}$ and $\Psi^{\a,0}_{*,E_3}$).
In the first case $|t-x|\le x$ and we get
the control by $[Df^{1+\theta}(x)]^{1/(1+\theta)}$. On the other hand, if $t > 2x$, then $z \simeq t$ and this part
of the maximal operator is controlled by $[Rf^{1+\theta}(x)]^{1/(1+\theta)}$.
\end{proof}

\subsection{Control of $\Psi^{\a,\b}_{*,F}$} \label{ssec:33}
We finally establish a control of $\Psi^{\a,\b}_{*,F}$ in terms of the special operators.
\begin{lema} \label{lem:ctrlPsiF}
Let $\a > -1$ and $-1/2 < \a+\b < 1/2$. Then, given any $\theta > 0$ such that $\frac{1}{1+\theta} < \a+\b+1/2$,
\begin{align*}
\Psi^{\a,\b}_{*,F}f(x) & \lesssim 	
\chi_{\{\b \ge 1\}}\big[H_{(\a+1/2+1-\b)(1+\theta)+1}\big(f^{1+\theta}\big)(2x)\big]^{\frac{1}{1+\theta}} \\
& \quad + \chi_{\{\b < 1\}}\big[H_{(\a+1/2)(1+\theta)+1}\big(f^{1+\theta}\big)(2x)\big]^{\frac{1}{1+\theta}} \\
& \quad +N_{2\a+2}f(x)
 +\big[ T_{(\a+1/2)(1+\theta)}\big( f^{1+\theta} \big)(x)\big]^{\frac{1}{1+\theta}},
\end{align*}
uniformly in $f \ge 0$ and $x > 0$.
\end{lema}

\begin{proof}
We will bound suitably the restricted maximal operators $\Psi^{\a,\b}_{*,F_1}$,
$\Psi^{\a,\b}_{*,F'_2}$ and $\Psi^{\a,\b}_{*,F''_2}$.

\noindent{\textbf{Region $F_1$.}}
Here $t \simeq x$. Then, by \eqref{relF1}, we have
\begin{equation*}
\Psi^{\a,\b}_{*,F_1}f(x)  \lesssim \sup_{x < t < 3x} \frac{(t-x)^{-\a-1/2}}{x^{2\a+\b+1}}\int_0^{t-x}
	(t-x-z)^{\a+\b-1/2} f(z) z^{2\a+1}\, dz.
\end{equation*}

By H\"older's inequality, the last integral is not bigger than
$$
\bigg(\int_0^{t-x} \big[ z^{\a+1/2}f(z)\big]^{1+\theta}\, dz\bigg)^{\frac{1}{1+\theta}}
\bigg( \int_0^{t-x} \big[ z^{\a+1/2} (t-x-z)^{\a+\b-1/2}\big]^{\frac{1+\theta}{\theta}}\, dz\bigg)^{\frac{\theta}{1+\theta}}.
$$
Assuming that $\b < 1$,
the second integral here is finite if and only if $\frac{1}{1+\theta} < \a+\b+1/2$.
If this is the case, it is comparable with $(t-x)^{(2\a+\b)\frac{1+\theta}{\theta}+1}$, as verified by changing the variable
$z \mapsto (t-x)z$. This gives
\begin{align} \nonumber
\Psi^{\a,\b}_{*,F_1}f(x) & \lesssim \sup_{x < t < 3x} 
	\frac{(t-x)^{\a+\b-1/2+\frac{\theta}{1+\theta}}}{x^{2\a+\b+1}}\bigg(\int_0^{t-x}\big[z^{\a+1/2}f(z)\big]^{1+\theta}\, dz
	\bigg)^{\frac{1}{1+\theta}} \\ \label{rel_b}
	& \lesssim \frac{1}{x^{\a+1/2+\frac{1}{1+\theta}}}\bigg(\int_0^{2x}\big[z^{\a+1/2}f(z)\big]^{1+\theta}\,
	dz\bigg)^{\frac{1}{1+\theta}} \\
	& = \big[H_{(\a+1/2)(1+\theta)+1}\big(f^{1+\theta}\big)(2x)\big]^{\frac{1}{1+\theta}}, \nonumber
\end{align}
where in the second relation we used the assumption $\a+\b+1/2 > \frac{1}{1+\theta}$.

When $\b \ge 1$ (then automatically $\a \le -1/2$) we use H\"older's inequality with the splitting of $z^{2\a+1}$ so that
$z^{\a+1/2+1-\b}$ is attached to $f$. This leads to the bound
$$
\Psi^{\a,\b}_{*,F_1}f(x) \lesssim \big[H_{(\a+1/2+1-\b)(1+\theta)+1}\big(f^{1+\theta}\big)(2x)\big]^{\frac{1}{1+\theta}}
$$
under the condition $\frac{1}{1+\theta} < \a+\b+1/2$.
The assumption $\beta \ge 1$ is used in the analogue of estimate \eqref{rel_b}, to assure that a power of $t-x$ occurring
there is non-negative.

\noindent{\textbf{Region $F'_2$.}}
In view of \eqref{relF2prim},
$$
\Psi^{\a,\b}_{*,F'_2}f(x) \lesssim \sup_{t \ge 3x} \frac{1}{t^{2\a+2}} \int_0^{(t-x)/2}f(z) z^{2\a+1}\, dz
	\le N_{2\a+2}f(x).
$$

\noindent{\textbf{Region $F''_2$.}}
Now we use \eqref{relF2} and get
$$
\Psi^{\a,\b}_{*,F''_2}f(x) \lesssim \sup_{t \ge 3x} \frac{1}{t^{2\a+\b+1}} \int_{(t-x)/2}^{t-x} (t+x-z)^{-\a-1/2}
	(t-x-z)^{\a+\b-1/2} f(z) z^{2\a+1}\, dz.
$$
By means of H\"older's inequality we can control the above integral by
$$
\bigg( \int_{(t-x)/2}^{t-x} \bigg[ \bigg( \frac{z}{t+x-z}\bigg)^{\a+1/2}f(z)\bigg]^{1+\theta} dz\bigg)^{\frac{1}{1+\theta}}
\bigg( \int_{(t-x)/2}^{t-x} \big[(t-x-z)^{\a+\b-1/2} z^{\a+1/2}\big]^{\frac{1+\theta}{\theta}}
	dz\bigg)^{\frac{\theta}{1+\theta}}.
$$
The second integral here is finite if and only if $\frac{1}{1+\theta} < \a+\b+1/2$. In this case
it is comparable with $(t-x)^{(2\a+\b)\frac{1+\theta}{\theta}+1}$. This implies, see \eqref{relF2bis},
\begin{align*}
\Psi^{\a,\b}_{*,F''_2}f(x) & \lesssim \sup_{t \ge 3x} \frac{(t-x)^{2\a+\b+\frac{\theta}{1+\theta}}}{t^{2\a+\b+1}}
	\bigg( \int_{(t-x)/2}^{t-x} \bigg[ \bigg( \frac{z}{t+x-z}\bigg)^{\a+1/2}f(z)\bigg]^{1+\theta}\, dz
	\bigg)^{\frac{1}{1+\theta}} \\
	& \lesssim \sup_{u \ge 2x} \bigg( \int_{u/2}^u \frac{z^{(\a+1/2)(1+\theta)-1}}{(u+x-z)^{(\a+1/2)(1+\theta)}}
	f^{1+\theta}(z)\, dz\bigg)^{\frac{1}{1+\theta}} \\
	& = \big[ T_{(\a+1/2)(1+\theta)}\big( f^{1+\theta} \big)(x)\big]^{\frac{1}{1+\theta}}.
\end{align*}

Now Lemma \ref{lem:ctrlPsiF} follows.
\end{proof}

\subsection{Proof of Theorem \ref{thm:main}} \label{ssec:34}
We distinguish several cases described by conditions on $\a$ and $\b$.
Altogether, they imply Theorem \ref{thm:main}.
In what follows we always assume that $\a > -1$ and $\a + \b > -1/2$, $1 < p < \infty$, and $x \in \mathbb{R}_+$.

\noindent \textbf{Case 1.} $\a + \b > 1/2$ and $-\b \notin \mathbb{N}$.
By Theorem \ref{thm:kermaj}, $M^{\a,\b}_*f(x) \lesssim \Phi^{\a,\b}_* f(x)$. Thus from Lemma \ref{lem:ctrlPhi}
we have the control
\begin{align*}
M_*^{\a,\b}f(x) & \lesssim Lf(x)+ H_{2\a+\b+1}f(4x)+ H_{2\a+2}f(2x) + N_{2\a+2}f(x)\\
					& \qquad + x^{\b}R\big(z^{-\b}f\big)(x) + T_{1-\b}f(x).
\end{align*}

The operator $L$ is bounded on $L^p(x^{\delta}dx)$ for any $\delta \in \mathbb{R}$.
The conditions for $L^p(x^{\delta}dx)$ boundedness of
the remaining operators controlling $M^{\a,\b}_{*}$ are as follows, see Section \ref{ssec:23}:
\begin{align*}
H_{2\a+\b+1}: & \qquad \delta < (2\a+\b+1)p-1, \\
H_{2\a+2}: & \qquad \delta < (2\a+2)p-1, \\
N_{2\a+2}: & \qquad -1 < \delta < (2\a+2)p-1, \\
f\mapsto x^{\b}R(z^{-\b}f)(x): & \qquad -\b p \le \delta,\\
T_{1-\b}: & \qquad \begin{cases}
								-1 < \delta, & \textrm{if}\; \b \ge 1, \\
								-\b \le \delta, & \textrm{if}\; \b \in (0,1), \\
								-\b p \le \delta, & \textrm{if}\; \b < 0. 
						\end{cases}
\end{align*}

We conclude that $M^{\a,\b}_*$ is bounded on $L^p(x^{\delta}dx)$ if $\delta$ satisfies
$$
- \min(1,\b,\b p) < \delta < \min\Big( (2\a+\b+1)p-1, (2\a+2)p-1 \Big),
$$
with the first inequality weakened when $\b < 1$.
Here
$$
\min(1,\b,\b p) = \begin{cases}
											1, & \textrm{if}\; \b \ge 1,  \\
											\b, & \textrm{if}\; \b \in (0,1), \\
											\b p, & \textrm{if}\; \b < 0,
									\end{cases}
$$
while
$$
\min\Big( (2\a+\b+1)p-1, (2\a+2)p-1 \Big) =
	\begin{cases}
		(2\a+\b+1)p-1, & \textrm{if}\; \b < 1, \\
		(2\a+2)p-1, & \textrm{if}\; \b \ge 1.
	\end{cases}
$$

\noindent \textbf{Case 2.} $\a+\b < 1/2$ and $-\b \notin \mathbb{N}$.
In view of Theorem \ref{thm:kermaj}, $M^{\a,\b}_*f(x) \lesssim \Psi^{\a,\b}_*f(x)$.
Consequently, from Lemmas \ref{lem:ctrlPsiE} and \ref{lem:ctrlPsiF} we get the control
\begin{align*}
M_*^{\a,\b}f(x) & \lesssim x^{\b}\big[ D(z^{-\beta}f)^{1+\theta}(x)\big]^{\frac{1}{1+\theta}} 
	+ x^{\b}\big[ R(z^{-\beta}f)^{1+\theta}(x)\big]^{\frac{1}{1+\theta}}\\
&\quad+ \chi_{\{\b < 1\}}\big[H_{(\a+1/2)(1+\theta)+1}\big(f^{1+\theta}\big)(2x)\big]^{\frac{1}{1+\theta}} \\
& \quad	+ \chi_{\{\b \ge 1\}}\big[H_{(\a+1/2+1-\b)(1+\theta)+1}\big(f^{1+\theta}\big)(2x)\big]^{\frac{1}{1+\theta}} \\
& \quad +N_{2\a+2}f(x)
 +\big[ T_{(\a+1/2)(1+\theta)}\big( f^{1+\theta} \big)(x)\big]^{\frac{1}{1+\theta}},
\end{align*}
with any fixed $\theta > 0$ satisfying $\frac{1}{1+\theta} < \a +\b + 1/2$.

Assuming that $1/p < \a+\b+1/2$, conditions for $L^p(x^{\delta}dx)$ boundedness of the operators controlling
$M^{\a,\b}_*$ are as follows, see Section \ref{ssec:23}:
\begin{align*}
f \mapsto x^{\b}\big[ D(z^{-\b}f)^{1+\theta}(x)\big]^{\frac{1}{1+\theta}}: & \qquad
		-\b p - 1 < \delta < (2\a+\b+1)p-1, \\
f \mapsto x^{\b}\big[ R(z^{-\b}f)^{1+\theta}(x)\big]^{\frac{1}{1+\theta}}: & \qquad
		-\b p  \le \delta, \\		
f \mapsto \chi_{\{\b < 1\}}\big[ H_{(\a+1/2)(1+\theta)+1}(f^{1+\theta})(x)\big]^{\frac{1}{1+\theta}}: & \qquad
						\delta < (2\a+\b+1)p-1,\\
f \mapsto \chi_{\{\b \ge 1\}}\big[ H_{(\a+3/2-\b)(1+\theta)+1}(f^{1+\theta})(x)\big]^{\frac{1}{1+\theta}}: & \qquad
	\delta < (2\a+2)p-1, \\
N_{2\a+2}: & \qquad -1 < \delta < (2\a+2)p-1, \\
f \mapsto \big[T_{(\a+1/2)(1+\theta)}(f^{1+\theta})(x)\big]^{\frac{1}{1+\theta}}: & \qquad \begin{cases}
								-\b p < \delta, & \textrm{if}\; \b < 0, \\
								-1 < \delta, & \textrm{if}\; \b > 0 \; \textrm{and} \; \a \le -1/2, \\
								\frac{-\b}{\a+\b+1/2} < \delta, & \textrm{if}\; \b > 0 \; \textrm{and} \; \a > -1/2.
						\end{cases}
\end{align*}
Treatment of the last operator will be explained in a moment; the cases of the other $\theta$-operators are simpler
and left to the reader.
In all these conditions we understand that given $\delta$ from the indicated range, there exists $\theta >0$ satisfying
$\frac{1}{1+\theta} < \a + \b +1/2$ such that the $L^p(x^{\delta}dx)$ boundedness holds.

Intersecting the above conditions we see that $M^{\a,\b}_*$ is bounded on $L^p(x^{\delta}dx)$ if
$1/p < \a+\b+1/2$ and any of the following three sets of conditions holds:
\begin{itemize}
\item[(a)]
$
\b \ge 1$ and $-1 < \delta < (2\a+2)p-1;
$
\item[(b)]
$
0 < \b < 1$ and  $\frac{-\b}{\max(0,\a+1/2)+\b} < \delta < (2\a+\b+1)p-1;
$
\item[(c)]
$
\b < 0$ and  $-\b p < \delta < (2\a+\b+1)p-1.
$
\end{itemize}

We now explain the analysis related to $T$ in greater detail.
Observe that, for $p> 1+\theta$, the operator
$f \mapsto \big[T_{(\a+1/2)(1+\theta)}(f^{1+\theta})(x)\big]^{\frac{1}{1+\theta}}$ is bounded on $L^p(x^{\delta}dx)$ if and
only if $T_{(\a+1/2)(1+\theta)}$ is bounded on $L^{\frac{p}{1+\theta}}(x^{\delta}dx)$.
If $\b < 0$ (then automatically $\a > -1/2$), consider $\theta > 0$ such that $1/p < \frac{1}{1+\theta} < \a+\b+1/2$
(recall that we assume $1/p < \a+\b+1/2$). Then also $\frac{1}{1+\theta} < \a+1/2$ and $T_{(\a+1/2)(1+\theta)}$
falls under Lemma \ref{lem:T}(a): it is bounded on $L^{\frac{p}{1+\theta}}(x^{\delta}dx)$ when
$$
\delta \ge \big[ (\a+1/2)(1+\theta)-1 \big]\frac{p}{1+\theta} = (\a+1/2)p - \frac{p}{1+\theta}.
$$
Choosing $\theta$ so that $\frac{1}{1+\theta}$ is sufficiently close to $\a+\b+1/2$ we can approach with the lower
bound for $\delta$ arbitrarily close to $(\a+1/2)p-(\a+\b+1/2)p = -\b p$.

Let now $\b >0$. If $\a \le -1/2$, then we choose any $\theta > 0$ such that $1/p < \frac{1}{1+\theta}<\a+\b+1/2$
and apply Lemma \ref{lem:T}(d). On the other hand, when $\a > -1/2$ (in fact $\a \in (-1/2,1/2)$), we consider $\theta >0$
such that $0 < \a+1/2 < \frac{1}{1+\theta} < \a+\b+1/2$. By Lemma \ref{lem:T}(c) the condition for the boundedness is
$\delta \ge (\a+1/2)(1+\theta)-1$. One can choose $\theta$ so that $\frac{1}{1+\theta}$ is arbitrarily close to
$\a+\b+1/2$, which covers all $\delta > \frac{\a+1/2}{\a+\b+1/2}-1 = \frac{-\b}{\a+\b+1/2}$.

\noindent \textbf{Case 3.} $\a + \b =1/2$ and $-\b \notin \mathbb{N}$.
By Theorem \ref{thm:kermaj}, $M^{\a,\b}_*f(x) \lesssim \Psi^{\a,\b-\varepsilon}_* f(x)$ with any fixed
$\varepsilon > 0$. From the analysis in Case 2, with $\b$ replaced by $\b-\varepsilon$, we infer that
$M^{\a,\b}_*$ is bounded on $L^p(x^{\delta}dx)$ if $1/p < \a+\b-\varepsilon+1/2$ and any of the following sets of conditions
holds:
\begin{itemize}
\item[(a)]
$
\b-\varepsilon \ge 1$ and $-1 < \delta < (2\a+2)p-1;
$
\item[(b)]
$
0 < \b-\varepsilon < 1$ and  $\frac{-(\b-\varepsilon)}{\max(0,\a+1/2)+\b-\varepsilon} < \delta < (2\a+\b-\varepsilon+1)p-1;
$
\item[(c)]
$
\b-\varepsilon < 0$ and  $-(\b-\varepsilon) p < \delta < (2\a+\b-\varepsilon+1)p-1.
$
\end{itemize}
Since $\varepsilon$ can be chosen arbitrarily small, the above conditions with $\varepsilon=0$
also imply $L^p(x^{\delta}dx)$ boundedness of $M^{\a,\b}_*$, as can be easily verified.

\noindent \textbf{Case 4.} $-\b \in \mathbb{N}$.
In view of Theorem \ref{thm:kermaj}, we have the bound $M^{\a,\b}_*f(x) \lesssim \Phi^{\a,\b}_{*,E}f(x)$ in case
$\a+\b \ge 1/2$, and $M^{\a,\b}_*f(x) \lesssim \Psi^{\a,\b}_{*,E}f(x)$ when $\a+\b < 1/2$.
Therefore, Lemmas \ref{lem:ctrlPhi} and \ref{lem:ctrlPsiE} provide the control
$$
M^{\a,\b}_*f(x) \lesssim
	\begin{cases}
		Lf(x) + H_{2\a+\b+1}f(4x) + x^{\b} R(z^{-\b}f)(x), & \textrm{if} \;\; \a+\b \ge 1/2, \\
		x^{\b}\big[ D(z^{-\b}f)^{1+\theta}(x)\big]^{\frac{1}{1+\theta}}
			+ x^{\b}\big[ R(z^{-\b}f)^{1+\theta}(x)\big]^{\frac{1}{1+\theta}}, & \textrm{if} \;\; \a+\b < 1/2,
	\end{cases}
$$
with any fixed $\theta > 0$ such that $\frac{1}{1+\theta} < \a+\b+1/2$.

Taking now into account conditions for $L^p(x^{\delta}dx)$ boundedness of the component operators involved, see
Cases 1 and 2 above, we conclude that $M^{\a,\b}_*$ is bounded on $L^p(x^{\delta}dx)$ if $1/p < \a +\b +1/2$
(this condition is automatically satisfied when $\a+\b \ge 1/2$) and
$$
-\b p \le \delta < (2\a+\b+1)p -1.
$$

The proof of Theorem \ref{thm:main} is complete. \qed

\subsection{Proof of Proposition \ref{prop:Mloc}} \label{ssec:35}
To begin with, observe that $M^{\a,\b}_{*,\textrm{tru}}$ arises from $M^{\a,\b}_*$ by restricting the kernel to the
region $E_1$.

When $\a + \b \ge 1/2$, we have $\Phi^{\a,\b}_{*,E_1}f(x) \lesssim Lf(x)$, see the proof of Lemma \ref{lem:ctrlPhi}.
Further, in case $\a+\b < 1/2$, $\Psi^{\a,\b}_{*,E_1}f(x) \lesssim x^{\b}[L(z^{-\b}f)^{1+\theta}(x)]^{1/(1+\theta)}$
with any fixed $\theta > 0$ such that $\frac{1}{1+\theta} < \a + \b +1/2$. This is implicitly contained in the proof
of Lemma \ref{lem:ctrlPsiE}. The more precise bounds of Theorem \ref{thm:kerest}(4),(5) lead easily to the
control $M^{\a,\b}_{*,\textrm{tru}}f(x) \lesssim Lf(x)$ in case $\a+\b=1/2$ and $-\b \notin \mathbb{N}$
(when $t < x/2$ the logarithm occurring in Theorem \ref{thm:kerest}(4) has no effect).
All this together with Theorem \ref{thm:kermaj} justifies \eqref{Mlocctrl}.

The remaining part of Proposition \ref{prop:Mloc} follows from the well-known fact that for any $\delta \in \mathbb{R}$,
$L$ is bounded on $L^p(x^{\delta}dx)$, $1<p<\infty$, and from $L^1(x^{\delta}dx)$ to weak $L^1(x^{\delta}dx)$.
\qed

\section{Counterexamples} \label{sec:counter}

Our aim in this section is to prove Proposition \ref{prop:sharp}.
This will be done by constructing suitable counterexamples, in the first step for certain auxiliary maximal operators.

\subsection{Counterexamples for auxiliary operators} \label{ssec:41}
For $\a > -1$, $\a+\b > -1/2$ and $t > 0$ consider the following integral operators, with positive kernels,
acting on functions on $\mathbb{R}_+$:
\begin{align*}
U_{t,1}^{\a,\b}f(x) & = \frac{x^{-2\a-\b}}{t^{2\a+2\b}} \int_{|t-x|}^{t+x} \Big( \big[(x+t)^2-z^2\big]
	\big[ z^2 - (x-t)^2\big] \Big)^{\a+\b-1/2} z^{1-\b} f(z)\, dz,\\
U_{t,2}^{\a,\b}f(x) & = \frac{x^{-\a-1/2}}{t^{2\a+2\b}} \int_{|t-x|}^{t+x} \big[ t^2-(x-z)^2 \big]^{\a+\b-1/2}
	z^{\a+1/2} f(z)\, dz, \\
\widetilde{U}_{t,2}^{\a,\b}f(x) & = \frac{x^{-\a-1/2}}{t^{2\a+2\b}} \int_{|t-x|}^{t+x} \big[ t^2-(x-z)^2 \big]^{\a+\b-1/2}
	\log\bigg( \frac{8xz}{(x+z)^2-t^2} \bigg) z^{\a+1/2} f(z)\, dz, \\
V_{t,1}^{\a,\b}f(x) & = \frac{\chi_{\{x<t\}}}{t^{2\a+2\b}} \int_0^{t-x} \big[ t^2-(x-z)^2\big]^{-\a-1/2}
	\big[t^2-(x+z)^2\big]^{\a+\b-1/2} z^{2\a+1} f(z)\, dz, \\
V_{t,2}^{\a,\b}f(x) & = \frac{\chi_{\{x<t\}}}{t^{2\a+2\b}} \int_0^{t-x} \big[t^2-(x-z)^2\big]^{\b-1} z^{2\a+1}f(z)\, dz, \\
\widetilde{V}_{t,2}^{\a,\b}f(x) & = \frac{\chi_{\{x<t\}}}{t^{2\a+2\b}} \int_0^{t-x} \big[t^2-(x-z)^2\big]^{\b-1}
	\log\bigg(2 \frac{t^2-(x-z)^2}{t^2-(x+z)^2} \bigg)z^{2\a+1}f(z)\, dz.
\end{align*}
These operators correspond to right-hand sides of the bounds (1)--(8) in Theorem \ref{thm:kerest}.
More precisely, in view of \eqref{id_tr}, $U_{t,1}^{\a,\b}$ matches the $E$ part in (1) and (3), (5), (6), (7), $U_{t,2}^{\a,\b}$
matches the $E$ part in (2), (5), (8), while $\widetilde{U}_{t,2}^{\a,\b}$ matches the $E$ part in (4).
Similarly, $V_{t,1}^{\a,\b}$ matches the $F$ part in (3), (5), (6), (8), $V_{t,2}^{\a,\b}$ matches the $F$ part in
(2), (5), (7), while $\widetilde{V}_{t,2}^{\a,\b}$ matches the $F$ part in (4).
Observe that
$$
U_{t,2}^{\a,\b}f(x) \lesssim \widetilde{U}_{t,2}^{\a,\b}f(x), \qquad
V_{t,2}^{\a,\b}f(x) \lesssim \widetilde{V}_{t,2}^{\a,\b}f(x), \qquad f \ge 0, \quad x,t > 0.
$$

Denote $U_{*,1}^{\a,\b}f = \sup_{t>0}|U_{t,1}^{\a,\b}f|$ and similarly for the remaining operators in question
replacing $t$ by $*$ in the subscript indicates the corresponding maximal operator.
\begin{lema} \label{lem:counter}
Let $\a > -1$, $\a+\b> -1/2$, $\delta \in \mathbb{R}$ and $1<p<\infty$.
\begin{itemize}
\item[(a)] The following are necessary conditions for each of $U_{*,1}^{\a,\b}$, $U_{*,2}^{\a,\b}$, $\widetilde{U}_{*,2}^{\a,\b}$
	to be well defined and bounded on $L^p(\mathbb{R}_+,x^{\delta}dx)$:
		\begin{itemize}
			\item[(a1)] $\delta < (2\a+\b+1)p-1$,
			\item[(a2)] $\frac{1}p < \a+\b+1/2$,
			\item[(a3)] $-\b p \le \delta$.
		\end{itemize}
\item[(b)] The following are necessary conditions for each of $V_{*,1}^{\a,\b}$, $V_{*,2}^{\a,\b}$, $\widetilde{V}_{*,2}^{\a,\b}$
	to be well defined and bounded on $L^p(\mathbb{R}_+,x^{\delta}dx)$:
		\begin{itemize}
			\item[(b1)] $\delta < (2\a+2)p-1$,
			\item[(b2)] $-1 < \delta$.
		\end{itemize}
\end{itemize}
\end{lema}

In the proof of Lemma \ref{lem:counter} we will need a simple technical result.
\begin{lema} \label{lem:iest}
Let $\gamma > -1$ and $0 \le A < B$. Then
$$
\int_A^B \Big( \big[ B^2-z^2\big] \big[z^2-A^2\big]\Big)^{\gamma} z \, dz
	= \frac{\Gamma(\gamma+1)^2}{2\Gamma(2\gamma+2)} \big(B^2-A^2\big)^{2\gamma+1}.
$$
\end{lema}

\begin{proof}
Changing the variable $(z^2-A^2)\slash (B^2-A^2)= s$ we immediately see that the integral in question is equal to
$$
\frac{1}2 \big(B^2-A^2\big)^{2\gamma+1} \int_0^1 s^{\gamma}(1-s)^{\gamma}\, ds.
$$
The last integral is the well known Euler beta integral.
The conclusion follows.
\end{proof}

\begin{proof}[Proof of Lemma \ref{lem:counter}]
We will construct suitable counterexamples.\\
\noindent \textbf{Part (a).}
Take $f(z)= \chi_{(1,2)}(z)$.
Then for large $x$ we have the bounds
\begin{align*}
U_{x-1,1}^{\a,\b} f(x) &
	\gtrsim \frac{x^{-2\a-\b}}{x^{2\a+2\b}}\int_1^{2} x^{2\a+2\b-1} (z-1)^{\a+\b-1/2}\, dz
		\simeq \frac{1}{x^{2\a+\b+1}}, \\
U_{x-1,2}^{\a,\b} f(x) &
	\gtrsim \frac{x^{-\a-1/2}}{x^{2\a+2\b}} \int_{1}^{2} x^{\a+\b-1/2} (z-1)^{\a+\b-1/2}\, dz
		\simeq \frac{1}{x^{2\a+\b+1}}.
\end{align*}
Thus a necessary condition for each of
$U_{*,1}^{\a,\b}$, $U_{*,2}^{\a,\b}$ and $\widetilde{U}_{*,2}^{\a,\b}$ to map $L^p(x^{\delta}dx)$ into itself is that the
function $x \mapsto x^{-2\a-\b-1}$ is in $L^p(x^{\delta}dx)$ for large $x$, which is equivalent to (a1).

Next, choose $f(z) = \chi_{(1,2)}(z) (z-1)^{-\a-\b-1/2}\slash \log\frac{2}{z-1}$. This function belongs to $L^p(x^{\delta}dx)$
when $\frac{1}p \ge \a+\b+1/2$. However, estimating similarly as above, for large $x$ we get
$$
U_{x-1,1}^{\a,\b} f(x)
	\gtrsim \frac{1}{x^{2\a+\b+1}} \int_1^{2} \frac{1}{(z-1)\log(\frac{2}{z-1})}\, dz = \infty,
$$
and in the same way $U_{x-1,2}^{\a,\b} f(x) = \infty$. This shows that (a2) is necessary for each of
$U_{*,1}^{\a,\b}$, $U_{*,2}^{\a,\b}$ and $\widetilde{U}_{*,2}^{\a,\b}$ to be well defined on $L^p(x^{\delta}dx)$.

Now let $f_N(z) = \chi_{(N-1,N+1)}(z) z^{\b}$ with $N$ large.
Then for $x \in (1,2)$ we have the estimates
$$
U_{N,1}^{\a,\b}f_N(x) \simeq \frac{1}{N^{2\a+2\b}} \int_{N-1}^{N+1}
		\Big( \big[(x+N)^2-z^2\big] \big[z^2-(x-N)^2\big]\Big)^{\a+\b-1/2} z \, dz \simeq 1,
$$
where the last relation follows from Lemma \ref{lem:iest}. Furthermore, still for $x\in (1,2)$,
\begin{align*}
U_{N,2}^{\a,\b}f_N(x) & \simeq \frac{1}{N^{2\a+2\b}} \int_{N-1}^{N+1} \big[ (N-x+z)(N+x-z)\big]^{\a+\b-1/2}
	N^{\a+1/2} N^{\b}\, dz \\
& \simeq \int_{N-1}^{N+1} (N+x-z)^{\a+\b-1/2}\, dz \simeq 1.
\end{align*}
All these bounds hold uniformly in $x$ and $N$.
On the other hand, when $\delta < -\b p$ it is straightforward to check that the norms of $f_N$ in $L^p(x^\delta dx)$
tend to zero as $N \to \infty$. This gives the necessity of (a3).

\noindent \textbf{Part (b).}
Observe that for $x> 0$ and $f \ge 0$ we have the bounds
\begin{align*}
V_{2x,1}^{\a,\b}f(x) & \simeq \frac{1}{x^{2\a+2\b}} \int_0^{x} \big[ 4x^2-(x-z)^2\big]^{-\a-1/2}
	\big[ 4x^2- (x+z)^2\big]^{\a+\b-1/2} z^{2\a+1} f(z)\, dz \\
  & \gtrsim \frac{1}{x^{2\a+2\b}} x^{-2\a-1}\int_0^{x/2} \big[ 4x^2-(x+z)^2\big]^{\a+\b-1/2} z^{2\a+1} f(z)\, dz \\
	& \simeq \frac{1}{x^{2\a+2}} \int_0^{x/2} z^{2\a+1} f(z)\, dz,\\
V_{2x,2}^{\a,\b}f(x) & \gtrsim \frac{1}{x^{2\a+2\b}} \int_0^{x/2} \big[4x^2-(x-z)^2\big]^{\b-1} z^{2\a+1} f(z)\, dz 
	 \simeq \frac{1}{x^{2\a+2}} \int_0^{x/2} z^{2\a+1} f(z)\, dz.
\end{align*}
Take now $f(z) = \chi_{(0,1)}(z)z^{-(\delta+1)\slash p}\slash \log\frac{2}{z}$. As easily verified, this function belongs to
$L^p(x^{\delta}dx)$. Hence, in view of the above estimates, a necessary condition for each of
$V_{*,1}^{\a,\b}$, $V_{*,2}^{\a,\b}$ and $\widetilde{V}_{*,2}^{\a,\b}$ to be well defined on $L^p(x^{\delta}dx)$ is that
the function $z \mapsto z^{2\a+1}z^{-(\delta+1)\slash p}\slash \log\frac{2}z$ is integrable at $0^+$, which is equivalent
to (b1).

Finally, consider $f(z)=\chi_{(0,1)}(z) z^{-\delta}$. Assume that $\delta \le -1$. Then $f \in L^p(x^{\delta}dx)$.
Further, choosing $t=2$, for $x \in (0,1)$ we have
\begin{align*}
V_{2,1}^{\a,\b}f(x) & \simeq \int_0^{2-x} \big[4-(x-z)^2\big]^{-\a-1/2} \big[ 4 - (x+z)^2\big]^{\a+\b-1/2} z^{2\a+1-\delta}
	\chi_{(0,1)}(z)\, dz \\
	& \gtrsim \int_0^{1/2} \big[4-(x-z)^2\big]^{-\a-1/2} \big[ 4 - (x+z)^2\big]^{\a+\b-1/2} z^{2\a+1-\delta}\, dz \\
	& \simeq \int_0^{1/2} z^{2\a+1-\delta}\, dz \simeq 1, \\
V_{2,2}^{\a,\b}f(x) & \simeq \int_0^{2-x} \big[ 4-(x-z)^2\big]^{\b-1} z^{2\a+1-\delta} \chi_{(0,1)}(z)\, dz 
	\simeq \int_0^{1} z^{2\a+1-\delta}\, dz \simeq 1.
\end{align*}
This implies that none of $V_{*,1}^{\a,\b}$, $V_{*,2}^{\a,\b}$ and $\widetilde{V}_{*,2}^{\a,\b}$ maps $L^p(x^{\delta}dx)$
into itself. The necessity of (b2) follows.

The proof of Lemma \ref{lem:counter} is complete.
\end{proof}

\subsection{Proof of Proposition \ref{prop:sharp}} \label{ssec:42}
Let $\a > -1$ and $\a+\b > -1/2$.
For $\varepsilon > 0$ denote
\begin{align*}
E_{\varepsilon} & = \big\{ (t,x,z) \in E : t - |x-z| < \varepsilon \sqrt{xz}\;\;
	\textrm{or} \;\; x+z -t < \varepsilon \sqrt{xz} \big\}, \\
F_{\varepsilon} & = \big\{ (t,x,z) \in F : t - (x+z) < \varepsilon \sqrt{xz}\;\;
	\textrm{or} \;\; t >  \varepsilon^{-1} \sqrt{xz} \big\}.
\end{align*}
From the results of \cite{CNR}, see \cite[Theorem 3.3]{CNR} and the accompanying comments,
it follows that for each pair $\a,\b$ there exists $\varepsilon = \varepsilon(\a,\b) > 0$ such that the kernel
$K_t^{\a,\b}(x,z)$ does not change sign in $E_{\varepsilon}$ and in $F_{\varepsilon}$ and, moreover,
$|K_t^{\a,\b}(x,z)|$ is comparable in $E_{\varepsilon}$ and in $F_{\varepsilon}$ with the corresponding expressions
on the right-hand side of the relevant bound from (1)--(8) in Theorem \ref{thm:kerest}.

Therefore the unboundedness results obtained for $U_{*,1}^{\a,\b}$, $U_{*,2}^{\a,\b}$, $\widetilde{U}_{*,2}^{\a,\b}$,
$V_{*,1}^{\a,\b}$, $V_{*,2}^{\a,\b}$, $\widetilde{V}_{*,2}^{\a,\b}$ in Lemma \ref{lem:counter} will also be valid for
$M_*^{\a,\b}$ (with $\a,\b$ suitably restricted according to the splitting in (1)--(8) in Theorem \ref{thm:kerest}) provided
that we assure that $(t,x,z)$ involved in the counterexamples are located either in $E_{\varepsilon}$ or in $F_{\varepsilon}$.

In the counterexamples from part (a) of the proof of Lemma \ref{lem:counter} one can consider either $x$ large enough (in the first
and the second counterexamples) or $N$ large enough (in the third counterexample) so that all $(t,x,z)$ involved are
contained in $E_{\varepsilon}$. In part (b), in the first counterexample one considers $x$ sufficiently large, while in
the second one $x$ in a sufficiently small neighborhood of $0$, so that all $(t,x,z)$ involved belong to $F_{\varepsilon}$.

We conclude that necessary conditions for $M_{*}^{\a,\b}$ to be well defined and bounded on $L^p(x^{\delta}dx)$ are
(a1), (a2), (a3) of Lemma \ref{lem:counter}, and in case $-\b \notin \mathbb{N}$ also (b1) and (b2).
However, since when $-\b \in \mathbb{N}$ conditions (b1) and (b2) are less restrictive than (a1) and (a3), respectively,
actually all (a1), (a2), (a3), (b1), (b2) are always the necessary conditions.
\qed

\begin{remark} \label{rem:sharp}
The discrepancy between the sufficient conditions in Theorem \ref{thm:main} and the necessary conditions
from Proposition \ref{prop:sharp} is related to region $F''_2$ and the special maximal operator $T_{\eta}$.
We do not know if the restriction of $M^{\a,\b}_*$ to $F''_2$ is always controlled by means of $T_{\eta}$ in an
optimal way. Neither we know whether the description of mapping properties of $T_{\eta}$ from Lemma \ref{lem:T} is sharp.
These issues remain to be investigated.
\end{remark}



\end{document}